\documentclass{amsart}
\usepackage{amssymb}
\usepackage{amsfonts}
\usepackage{amssymb}
\usepackage{amsmath, accents}
\usepackage{amsthm}
\usepackage{enumerate}
\usepackage{tabularx}
\usepackage{centernot}
\usepackage{mathtools}
\usepackage{stmaryrd}
\usepackage{amsthm,amssymb}
\usepackage{etoolbox}
\makeatletter
\@ifundefined{c@unusedcounter}{}{
  \@addtoreset{unusedcounter}{section}
}
\makeatother
\usepackage{tikz}
\usepackage{amssymb}
\usetikzlibrary{matrix}
\usepackage{tikz-cd}
\usepackage{tikz}
\usepackage{marginnote}
\definecolor{mygray}{gray}{0.85}
\usepackage[backgroundcolor=mygray,colorinlistoftodos,prependcaption,textsize=small]{todonotes}
\usepackage{xcolor}
\newcommand{\mbf}{\mathbf}
\newcommand{\msc}{\mathcal}

\newcommand{\op}{\operatorname}
\newcommand{\join}{\vee}
\renewcommand{\Join}{\bigvee}
\newcommand{\meet}{\wedge}
\newcommand{\Meet}{\bigwedge}
\newcommand{\la}{\langle}
\newcommand{\ra}{\rangle}
\renewcommand{\leq}{\leqslant}
\renewcommand{\geq}{\geqslant}
\renewcommand{\nleq}{\nleqslant}
\newcommand{\vare}{\varepsilon}

\newcommand{\ol}{\overline}
\newcommand{\ul}{\underline}
\newcommand{\onto}{\twoheadrightarrow}
\newtheorem{thm}{Theorem}[section]
\newtheorem{theorem}[thm]{Theorem}
\newtheorem{cor}[thm]{Corollary}
\newtheorem{lm}[thm]{Lemma}
\newtheorem{notation}[thm]{Notation}

\newtheorem{prob}[thm]{Problem}

\newtheorem{defn}[thm]{Definition}

\newtheorem{fact}[thm]{Fact}

\date{\today}

\begin{document}

\author{J.\,B. Nation}

\address{Department of Mathematics, University of Hawaii, Honolulu, HI 96822.}
\email{jb@math.hawaii.edu}

\author{Gianluca Paolini}

\address{Department of Mathematics ``Giuseppe Peano'', University of Torino, Via Carlo Alberto 10, 10123, Italy.}
\email{gianluca.paolini@unito.it}

\begin{abstract}  
We continue our work on the model theory of free lattices, solving two of the main open problems from our first paper on the subject.
Our main result is that the universal (existential) theory of infinite free lattices is decidable. Our second main result is a proof that finitely generated free lattices are positively distinguishable, as for each $n \geq 1$ there is a positive $\exists \forall$-sentence true in $\mbf F_n$ and false in $\mbf F_{n+1}$. Finally, we show that free lattices are first-order rigid in the class of finitely generated projective lattices, and that a projective lattice has the same existential (universal) theory of an infinite free lattice if and only if it has breadth $> 4$ (i.e., a single existential sentence is sufficient).

\end{abstract}

\thanks{Research of Gianluca Paolini was  supported by project PRIN 2022 ``Models, sets and classifications", prot. 2022TECZJA, and by INdAM Project 2024 (Consolidator grant) ``Groups, Crystals and Classifications''.}

\title[Elementary properties of free lattices II]{Elementary properties of free lattices II: decidability of the universal theory}

\maketitle  

\section{Introduction}

The model theory of free objects, particularly free groups, has produced many deep and beautiful results. Building on this tradition, in \cite{NaPa14} we began investigating the model theory of free lattices, focusing on questions of decidability and elementary equivalence. Three key questions from~\cite{NaPa14} have remained open:
\begin{itemize}
    \item Is the $\forall$-theory of an infinite free lattice decidable?
    \item Can finitely generated free lattices be distinguished by positive first-order logic (i.e., first-order logic without negations)?
    \item Is each finitely generated free lattice the unique finitely generated model of its first-order theory (i.e., is it first-order rigid)?
\end{itemize}

In this paper, we resolve the first two questions affirmatively and provide new insights on the third. The decidability of the $\forall$-theory of free lattices represents our most significant contribution and will be our primary focus. We start by introducing this problem in more detail. 
Let $\mathbf{F}_{\lambda}$ denote the free lattice on $\lambda$-many generators. 
The free lattice $\mathbf{F}_{\lambda}$ is infinite whenever $\lambda \geq 3$.
It is well known that all countable infinite free lattices are bi-embeddable, 
that is, $\mbf F_m$ and $\mbf F_n$ are bi-embeddable whenever $3 \leq m, n \leq \aleph_0$.
Moreover, easily, for any uncountable cardinal $\kappa$, $\mathbf{F}_{\aleph_0}$ is an elementary substructure of $\mathbf{F}_{\kappa}$.   
Consequently, all infinite free lattices 
%(i.e., those free lattices with at least three generators)
share the same universal (and thus also the same existential) theory. This uniformity justifies our discussion of ``the $\forall$-theory of an infinite free lattice'' as a well-defined concept. To the best of our knowledge, nothing is known about this problem, apart from the decidability of the word problem for free lattices, which dates back to Whitman's fundamental paper \cite{PMW1941} from 1941. Actually, at this point, a historical remark seems in order, as this will be relevant for our proof strategy toward decidability. Skolem, as part of his 1920 paper which proves the L\"owenheim-Skolem Theorem \cite{ThS}, solved the word problem not only for free lattices, but for finitely
presented lattices as well. However, by the time the great awakening of lattice theory occurred in the 1930’s, his solution had been forgotten. Thus, Whitman’s 1941 construction of free lattices became the standard reference on the subject. It was not until 1992 that Stan Burris rediscovered Skolem’s solution. In fact, to be precise, Skolem also proved another result, i.e., that the universal Horn theory of lattices is decidable (see e.g. \cite[Chapter~9]{hodges} on Horn sentences). This is the starting point of our algorithm for deciding the truth value of a universal (equiv. existential) sentence in free lattices, which we now outline.

\medskip    Recall that every first-order sentence can be put into prenex form, moving the quantifiers to the front.  
%% algorithm for that is standard: write in terms of not, and, or (removing implies) then   not forall = exists not    & dually
Then the statement can be written in a disjunctive normal form.
%% standard boolean logic
We consider sentences $\psi(\mbf x)$ whose prenex disjunctive normal form is:
\[  \exists x_1 \dots \exists x_m \  \psi_1(\mbf x)  \,\op{OR}\, \dots \,\op{OR}\, \psi_p(\mbf x), \]
where $\mbf x = (x_1, ..., x_m)$ and each $\psi_j(\mbf x)$ is a (finite) conjunction of lattice literals, i.e., $s(\mbf x) \leq t(\mbf x)$ or $u(\mbf x) \not\leq v(\mbf x)$
with $s$, $t$, $u$, $v$ lattice terms in the variables $\mbf x = (x_1, ..., x_m)$.   
%%(More generally, allow $\leq$, $\nleq$, $=$, and $\neq$ .)
%% atomic formulas and their negations
The sentence $\psi(\mbf x)$ holds in a free lattice $\mbf F(Y)$ if and only if $\exists \mbf x \ \psi_j(\mbf x)$ holds in $\mbf F(Y)$ for some $1 \leq j \leq p$.  
%%what it means to "hold" will be defined below.
Now, $\psi_j(\mbf x)$ is like a lattice presentation, except that negations can be included.  
%So we follow Section 2-3.1 of Freese and Nation \cite{FN_STA} on finitely presented lattices, which in turn are based on Skolem \cite{ThS}.
%%The skolem reference shortens the title, has no pages, no date;
%%there is a translation on my webpage that could be referenced.
Modifying Skolem's algorithm slightly, we will convert each $\psi_j(\mbf x)$ into a \emph{relational quasilattice with negations}, and use that to produce a \emph{partially defined lattice $\mbf S$ with negations}.  From there we will construct the \emph{partial completion} $\op{PC}(\mbf S)$ (a partially defined lattice in the sense of \ref{def:partdeflat}), and from it a \emph{finitely presented lattice} $\op{FP}(\mbf S)$. %We introduce a useful definition.
%
% \begin{defn}\label{def_occurs} We say that \emph{$\mbf S$ occurs in $\mbf F(Y)$} if there is a (partial lattice) homomorphism 
%$\eta : \mbf S \to \mbf F(Y)$ preserving negations. 
%\end{defn}
%
%Notice that $\mbf S$ occurs in the free lattice $\mbf F(Y)$ if and only if $\mbf F(Y)$ satisfies $\psi_j(\mbf x)$, where $\psi_j(\mbf x)$ is the formula used to %construct $\mbf S$.
%
 The crucial observation is then that the partially defined lattice $\mbf S$ occurs in a free lattice $\mbf F(Y)$ (equivalently, $\mbf F(Y)$ satisfies $\psi_j(\mbf x)$) if and only if there is a retraction of $\op{FP}(\mbf S)$ onto a projective lattice $\mbf G$ that also contains $\mbf S$ (Lemma~\ref{lm:basic}).  Whether or not this retraction of $\op{FP}(\mbf S)$ exists can be tested in the 
 congruence lattice of $\mbf S$, which is finite.
There are various properties to be tested for this, making the criteria for the retraction rather technical, but finite, and that makes our problem decidable, thus establishing the main theorem of our paper:

    \begin{theorem}\label{main_theorem} The universal theory of infinite free lattices is decidable.   
    \end{theorem}

    We recall that by {\em projective lattice} we mean a projective object in the category of lattices. In analogy with the theory of hyperbolic groups, we say that a projective lattice is non-elementary if it embeds $\mbf{F}_3$. Now, since every countable, non-elementary projective lattice is bi-embeddable with $\mbf{F}_3$, every non-elementary projective lattice has the same $\forall$-theory as $\mbf{F}_3$. So naturally we ask if we can separate the $\forall$-theories of elementary and non-elementary projective lattices. We answer this question with a result of independent interest:

    \begin{thm}\label{intermediate_theorem} Let $L$ be a lattice satisfying Whitman's Axiom $(W)$, then the following are equivalent:
    \begin{enumerate}[(1)]
        \item $L$ embeds $\mbf{F}_3$;
        \item $L$ has breadth $> 4$.
    \end{enumerate}
 \end{thm}

    \begin{cor}\label{intermediate_theorem_cor} A projective lattice has the same $\forall$-theory (equiv.~$\exists$-theory) of an infinite free lattice if and only if it has breadth $> 4$.
\end{cor}

Notice that in both \ref{intermediate_theorem} and \ref{intermediate_theorem_cor} the number $4$ cannot be replaced with $3$ as the lattice of Figure 1.1 of \cite{FJN} shows. Notice also that many infinite projective lattices which do not embed $\mbf{F}_3$ are known in the literature, the most famous example is probably the lattice $\mathbf{F}(\mbf{2+2})$ (cf.~\cite[pg. 105]{FJN}), i.e., the free lattice over a partial order set consisting of two independent $2$-element chains. This lattice is diagrammed in Figure~\ref{fig:fl22}; using the diagram one can easily verify that this lattice has breadth $\leq 4$ (and so it does not embed~$\mathbf{F}_3$). 

    \begin{figure}
    \centering
    \includegraphics[width=0.6\textwidth]{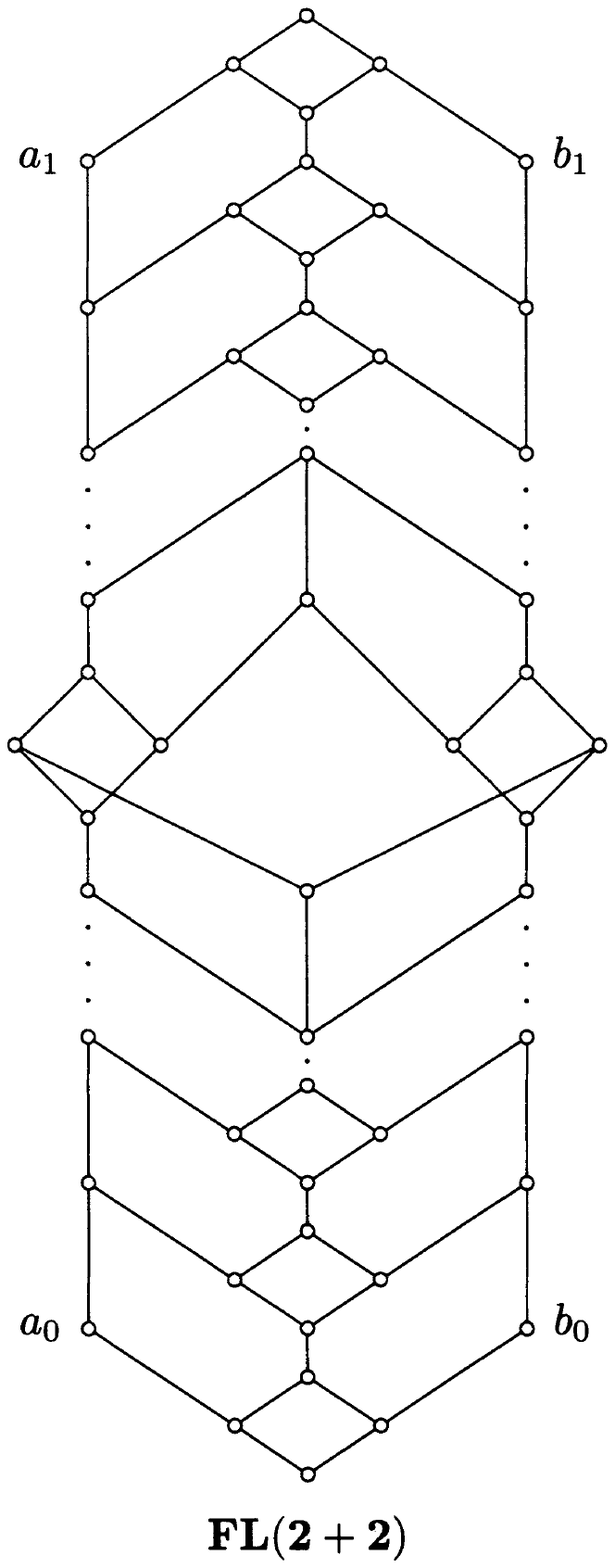}
    \caption{The free lattice over $\mbf{2+2}$}
    \label{fig:fl22}
\end{figure}

    We now move to the second main theorem of our paper, which is concerned with positive logic. In \cite{NaPa14} we showed that $\mbf{F}_3$ and $\mbf{F}_4$ were positively distinguishable, but a general solution for arbitrary $n$ has eluded us. This made us naturally wonder if this was a limitation of our methods or if there was an intrinsic reason for this. In the subsequent theorem we solve this problem, showing that the situation is uniform in $n$. Interestingly, we will show in Section~\ref{positive_other_varieties} that our proof actually applies to many other varieties of lattices (see there for details on which varieties).

      \begin{theorem}\label{second_theorem} Finitely generated free lattices are positively distinguishable, as for each $n \geq 1$ there is a positive $\exists \forall$-sentence true in $\mbf F_n$ and false in $\mbf F_{n+1}$.  
    \end{theorem}

     Having solved two of the main open problems from \cite{NaPa14} we now move to the third problem mentioned at the beginning of the introduction. In this respect, in \cite{NaPa14} we showed that every model of $\mathrm{Th}(\mbf{F}_n)$ admits a canonical projection on the profinite-bounded completion $\mbf{H}_n$ of $\mbf{F}_n$, and that this lattice is isomorphic to the Dedekind-MacNeille completion of $\mbf{F}_n$. From this we deduced that $\mbf{F}_n$ and $\mbf{H}_n$ are not elementary equivalent. This gave us some information on models of $\mbf{F}_n$ (and so also on finitely generated models of $\mbf{F}_n$), but it did not allow us for a full solution to the problem of first-order rigidity of finitely generated free lattices. This problem appears to us to be untractable at the moment in its general form. On the other hand, we will show that if we restrict to projective lattices, then we can obtain a strong form of first-order rigidity. Recall from above that every non-elementary projective lattice shares the same $\forall$-theory of an infinite free lattice. In contrast:

    \begin{theorem}\label{first-order_rig_th} Free lattices are first-order rigid in the class of finitely generated projective lattices, i.e., if $\mbf{L}$ is a finitely generated projective lattice and there is $3 \leq n < \omega$ such that $\mbf{L}$ is elementary equivalent to $\mbf{F}_n$ then $\mbf{L} \cong \mbf{F}_n$.
    \end{theorem}

    Thus, the question of first-order rigidity of finitely generated free lattice reduces to the following fundamental question which we shall leave open:

    \begin{prob} Is there a non-projective (equivalently, non-bounded) finitely generated lattice $\mbf{L}$ elementary equivalent to a finitely generated free lattice?
    \end{prob}

    The other major open problem on the model theory of free lattices is:

    \begin{prob} Is the full first-order theory of an infinite free lattice decidable?
    \end{prob}

    A few words on the structure of the paper. Section~2 is devoted to proving the fundamental lemmas that form the basis of our strategy for the decidability of the universal theory of free lattices, while in Section~3 we proceed with its implementation, thus establishing our Main Theorem, i.e., Theorem~\ref{main_theorem}. In Section~4 we prove Theorem~\ref{intermediate_theorem}. In Section~5 we prove our main theorem
    on positive distinguishability, i.e., Theorem~\ref{second_theorem}. Finally, in Section~6 we prove Theorem~\ref{first-order_rig_th}.

\subsection{Notation}

Throughout we use boldface to denote $k$-tuples, 
so that for example $\mbf x = (x_1, \dots, x_k)$, etc.
The notation $\downarrow\! x$ denotes $\{ u \in L : u \leq x \}$, and dually $\uparrow\! x$ denotes $\{ v \in L : v \geq x \}$. Following \cite{FJN}, we denote lattices with boldface letters.

\section{The strategy for decidability of the universal theory}

In this section, we lay out our strategy for a proof of decidability of the universal theory of free lattices. For the benefit of the reader, we repeat the outline of our strategy from the introduction. Every first-order sentence can be put in prenex form, moving the quantifiers to the front.  
Then the statement can be written in a disjunctive normal form.
We consider sentences $\psi(\mbf x)$ whose prenex disjunctive normal form is:
\[  \exists x_1 \dots \exists x_m \  \psi_1(\mbf x)  \,\op{OR}\, \dots \,\op{OR}\, \psi_p(\mbf x), \]
where $\mbf x = (x_1, ..., x_m)$ and each $\psi_j(\mbf x)$ is a (finite) conjunction of lattice literals, i.e., $s(\mbf x) \leq t(\mbf x)$ or $u(\mbf x) \not\leq v(\mbf x)$
with $s$, $t$, $u$, $v$ lattice terms in the variables $\mbf x = (x_1, ..., x_m)$.   
%%(More generally, allow $\leq$, $\nleq$, $=$, and $\neq$ .)
%% atomic formulas and their negations
The sentence $\psi(\mbf x)$ holds in a free lattice $\mbf F(Y)$ if and only if $\exists \mbf x \ \psi_j(\mbf x)$ holds in $\mbf F(Y)$ for some $1 \leq j \leq p$.  
%%what it means to "hold" will be defined below.

Now $\psi_j(\mbf x)$ is like a lattice presentation, except that negations can be included.  
So we follow Section 2-3.2 of Freese and Nation \cite{FN_STA} on finitely presented lattices, which in turn are based on Skolem \cite{ThS}.
%%The skolem reference shortens the title, has no pages, no date;
%%there is a translation on my webpage that could be referenced.
Modifying Skolem only slightly, we will convert each $\psi_j(\mbf x)$ to a \emph{relational quasilattice with negations}, and use that to produce a \emph{partially defined lattice $\mbf S$ with negations}.  From there we will construct the \emph{partial completion} $\op{PC}(\mbf S)$ (a partially defined lattice in the sense of \ref{def:partdeflat}), and from it the \emph{finitely presented lattice} $\op{FP}(\mbf S)$. We introduce a useful definition.

 \begin{defn}\label{def_occurs} Let $\mbf S$ be a partially
defined lattice. We say that \emph{$\mbf S$ occurs in $\mbf F(Y)$} if there is a (partial lattice) homomorphism 
$\eta : \mbf S \to \mbf F(Y)$ preserving negations. 
\end{defn}

Notice that $\mbf S$ occurs in the free lattice $\mbf F(Y)$ if and only if $\mbf F(Y)$ satisfies $\psi_j(\mbf x)$, where $\psi_j(\mbf x)$ is the formula used to construct $\mbf S$.

\smallskip 

 The crucial observation is then that the partially defined lattice $\mbf S$ occurs in a free lattice $\mbf F(Y)$ (equivalently, $\mbf F(Y)$ satisfies $\psi_j(\mbf x)$) if and only if there is a retraction of $\op{FP}(\mbf S)$ onto a projective lattice $\mbf G$ that also contains $\mbf S$ (Lemma~\ref{lm:basic}).  Whether or not this retraction of $\op{FP}(\mbf S)$ exists can be tested in the 
 congruence lattice of $\mbf S$. 
 The criteria for the retraction to exist are rather technical, but finite, and that makes our problem decidable, thus establishing the main theorem of our paper, i.e., Theorem \ref{main_theorem}.

 \medskip

This section is devoted to proving all the necessary lemmas on which our strategy relies, while in the next one we will prove that the existence of the retraction of $\op{FP}(\mbf S)$ can be tested in the finite lattice $\op{Con} \,\mbf S$. This section is divided into four subsection. In the first one we review Skolem's algorithm to decide whether an existential sentence is true in all lattices, and use it to construct the lattice $\op{PC}(\mbf S)$. The second subsection is devoted to the construction of the lattice $\op{FP}(\mbf S)$. The third section presents the necessary background on bounded and projective lattices, toward a proof in the fourth subsection of our ``Basic Lemma'' (i.e., Lemma~\ref{lm:basic}) connecting satisfiability of the formula $\psi_j(\mbf x)$ in $\mbf{F}(Y)$ and the existence of a retraction of $\op{FP}(\mbf S)$ onto a projective~lattice~$\mbf{G}$.

\subsection{Skolem's algorithm and the lattice $\op{PC}(\mbf S)$}

So let us begin.  Recall that \emph{lattice terms} over an alphabet $X$ are defined recursively:
\begin{itemize}
    \item each $x \in X$ is a term;
    \item if $r$ and $s$ are terms, then $(r \join s)$ and $(r \meet s)$ are terms;
    \item only expressions obtained by these rules are terms.
\end{itemize}
In particular, we regard $\join$ and $\meet$ as binary operations; with minimal adjustment one can treat them as finitary operations.
%%  if we need to, include here def'n of rank and depth of a term

The set of subterms of a term is defined as usual:
if $t$ is a variable then $\{t\}$ is its set of subterms, and if $t = t_1 \join t_2$ or $t = t_1 \meet t_2$ then the set of subterms is the union of $\{t\}$ and the subterms of $t_i$ for $i = 1,2$.

Dropping the subscript ``$j$" on $\psi$, we are given a formula $\psi(\mbf x)$ which is a conjunction of lattice literals, and thus can be written in the form
\[ \psi(\mbf x) \qquad s_1 \leq t_1 \ \&\  \dots \ \&\  s_k \leq t_k \ \&\  u_1 \nleq v_1 \ \&\  \dots u_\ell \nleq v_\ell  \]
where $s_i$, $t_i$, $u_i$, $v_i$ are all terms in $X = \{x_1, ..., x_m\}$, where we recall that $\mbf x = (x_1, ..., x_m)$.
Let us record such a \emph{lattice presentation with negations} as a triple $\la X, R_+, R_- \ra$ where $X$ is the set of variables and we let:
\begin{align*}
    R_+ &= \{ s_1 \leq t_1, \dots, s_k \leq t_k \}, \\
    R_- &= \{ u_1 \nleq v_1, \dots, u_\ell \nleq v_\ell \}. 
\end{align*}

\begin{defn}\label{def_quasilattice}
A set $U$ with three relations $\leq$, $\join$, and $\meet$ of arities 2, 3 and~3, respectively,
is a \emph{relational quasilattice} if it satisfies
the following axioms (i)--(v) and the duals of (iii)--(v): 
\begin{enumerate}[(i)]
\item $\leq $ is reflexive.
\item $\leq $ is transitive.
\item If $(x,y,z) \in \meet$, then $(z,x)$ and $(z,y)$ are in $\leq$. 
\item If $(x,y,z) \in \meet$ and both $(u,x)$ and $(u,y)$ are in $\leq$ then $(u,z)$ is in $\leq$.
\item  If $(x,y,z) \in \meet$ and 
$(x,x')$, $(x',x),$ 
$(y,y')$, $(y',y),$ 
$(z,z')$, $(z',z)$ are all in~$\leq$, then $(x',y',z') \in \meet$.   
%%insert jbn\item For all $x$, $y$ there is a $z$ with $(x,y,z) \in \mm$.    %%insert jbn
\end{enumerate}
\end{defn}
Skolem has a part (vi), which was copied into \cite{FN_STA}, but is never used and best not included.  
It says $(\forall x)(\forall y)(\exists z)$ such that $(x,y,z) \in \meet$.  %%as you can see in blue, mea culpa, dumb

Now it is possible that a lattice presentation with negations is inconsistent within lattice theory.  
This would happen when:
\begin{equation*}
(\dagger) \qquad\quad
\&_{1 \leq i \leq k} \; s_i \leq t_i \implies u_j \leq v_j 
\end{equation*}
for some pair with $u_j \nleq v_j$ in $R_-$.
An example is the presentation
\[  a \leq b \join c      \quad\text{AND}\quad    d \leq a \join b \quad\text{AND}\quad   d \nleq b \join c  \] 
%The statement is that any such $\psi_j$ can be removed and replaced with the boolean value \emph{false}.
This is inconsistent with lattice theory {\em itself}, and not only with the theory of free lattices.  

Skolem's algorithm, Theorem~\ref{skolem} below, tests for inconsistency in the presentation.
The algorithm creates a relational structure $(U,\leq,\join,\meet)$,
where $U$ is the set of all subterms of the terms occurring in $R_+ \cup R_-$.  All three relations $\leq, \wedge, \vee$ are initially empty.

\begin{thm}[Skolem's Algorithm]\label{skolem}
The following polynomial time algorithm 
decides if the implication $(\dagger)$ is valid. Let
$U$ be the set of all subterms of the terms $s_i$, $t_i$, $u_j$, $v_j$ in $R_+ \cup R_-$. 
\begin{enumerate}[(a)]
\item For $i = 1,\ldots,k$, add $(s_i,t_i)$ to~ $\leq$.
\item If $r = r_1 \join r_2$ is a subterm, then add $(r_1,r_2,r)$ to $\join$.
\item If $r = r_1 \meet r_2$ is a subterm, then add $(r_1,r_2,r)$ to $\meet$.
\item Close $(U,\leq,\join,\meet)$ under \emph{(i)--(iv)} and the duals of \emph{(iii)}~and~\emph{(iv)}. 
%\item[(e)] If $(u_j,v_j)$ is in $\le$ then $(\dagger)$ is
%true; otherwise not.
\end{enumerate}
Once we are done constructing $U$ we check if in $U$ we have that $(u_j,v_j)$ is in $\leq$; if so, then $(\dagger)$ is
true; otherwise not.
\end{thm}

Of course, when the implication is true for any $j$, the presentation is inconsistent.
The proof is given as Theorem~2-3.2 in \cite{FN_STA}.
We need to extract some information from that proof, as it contains the construction of the partially defined lattice from which we get the partial completion of $\op{PC}(\mbf S)$, which is the objet of interest of the present section.

After completing part (d), the relation $\leq$ becomes a
quasiorder by axioms~(i) and~(ii) from \ref{def_quasilattice}. 
Moreover, axiom~(v) and its dual hold for $(U, \leq, \join, \meet)$.
Let $\equiv$ be the equivalence
relation associated with this quasiorder, where $u\equiv v$ if $(u,v)$
and $(v,u)$ are in $\leq$. Let $\ol u$ be the equivalence
class of~$u$ and $\ol U = \{\ol u : u \in U\}$. Of
course $\ol U$ is an ordered set. Also note that if 
$(a,b,c)\in\join$ then by axioms~(iii) and~(iv), $\ol c$ is the least upper bound of $\ol a$ and $\ol b$ under the order of $\ol U$. 

\begin{defn}\index{Partially defined lattice}%
\label{def:partdeflat}%
A \emph{partially defined lattice} ($\mathrm{PDL}$, for short)
is a partially ordered set $(P,\leq)$ together with two partial
functions, $\Join$ and $\Meet$, from subsets of $P$ into $P$ such that
if $p = \Join S$ then $p$ is the least upper bound of $S$ in
$(P,\leq)$, and dually. We use $(P,\leq,\Join,\Meet)$ to denote this structure.
\end{defn}

By the above remarks, $(\ol U, \leq, \Join,\Meet)$ is a partially defined lattice where $\Join$ is given by the rule:
for each $(a,b,c) \in \join$ we have
$\Join\{\ol a, \ol b\} = \ol c$, and dually for $\Meet$. 

Lemma \ref{lm:embed} below shows that a partially defined lattice is embedded into its ideal lattice, in this case $\op{Idl}(\ol U,\le,\Join,\Meet)$. 
Since $\ol s_i \leq \ol t_i$ for $i = 1,\ldots,k$, if $\ol u_j \nleq \ol v_j$ then this substitution into $\op{Idl}(\ol U, \leq, \Join,\Meet)$ witnesses the failure of~$(\dagger)$. 

Let us now rename $\mbf S = (\ol U, \leq, \Join,\Meet)$, and since we still need to keep track of the non-inclusions in $R_-$, call the pair $\la \mbf S, R_- \ra$ the \emph{partially defined lattice with negations} determined by the presentation $\la X, R_+, R_- \ra$ (assuming it is consistent, i.e., that it has passed Skolem's test). Note that the construction of $\op{PC}(\mbf S)$ which we will introduce below contains terms and subterms from both $R_+$ and $R_-$, and so in particular it depends on both $R_+$~and~$R_-$. This point is slightly obscured by the notation $\mbf S$; thus we invite the reader to keep this in mind despite the notation.
%In fact it will be more correct to write $\op{FP}(\mbf S)$ as $\op{FP}(\op{PC}(\mbf S))$, since this lattice is constructed with $\op{PC}(\mbf S)$ as a starting point, and the partially defined lattice $\op{PC}(\mbf S)$ has embedded in it all the information from $\la \mbf S, R_- \ra$, but for ease of notation we follow the notation $\op{FP}(\mbf S)$.

    \begin{defn}\label{S_homo} Given $\la \mbf S, R_- \ra$ and a partially defined lattice $\mbf K$, a map $h: \mbf S \to \mbf K$ is an \emph{$\mbf S$-homomorphism} if it is a partially defined lattice (PDL) homomorphism such that $h(u) \nleq h(v)$ for all $u \nleq v$ in $R_-$.
We say that \emph{$\mbf S$ occurs in $\mbf K$} if there is an $\mbf S$-homomorphism $h : \mbf S \to \mbf K$. 
\end{defn}

An \emph{ideal} $I$ in a partially defined lattice  
$\mbf S$ is a subset of $\mbf S$ such that if $a\in I$ 
and $b \leq a$ then $b\in I$, and if $a_1,\ldots,a_k$ are in $I$ and $a = \Join a_i$ is defined then $a \in I$.
%%could write this just for binary
It is worth pointing out that these two rules may have to be applied repeatedly to find the ideal generated by a set. 
The set of all ideals of $\mbf S$ including the empty ideal, ordered by set inclusion, forms a lattice denoted $\op{Idl}(\mbf S)$.

\begin{lm} \label{lm:embed}
The map $s \mapsto \,\downarrow\! s$ embeds $\mbf S$ into $\op{Idl}(\mbf S)$, preserving the order (and its negation) and all the defined joins and meets. 
\end{lm}
\begin{proof}
    If $x$, $y \leq s$ then $x \join y \leq s$ by axiom~(iii) of relational quasilattices.  If $s \meet t$ is defined, say $s \meet t = w$, then $\downarrow\! s \,\cap \downarrow\! t = \,\downarrow\! w$ by axiom~(iv).
\end{proof}

Dually, a \emph{filter} $F$ in $\mbf S$ is a subset of $\mbf S$ such that if $a\in F$ 
and $b \geq a$ then $b\in F$, and if $a_1,\ldots,a_k$ are in $F$ and $a = \Meet a_i$ is defined then $a \in F$.
%%could write this just for binary
The set of all filters of $\mbf S$ including the empty filter, \emph{ordered by reverse set inclusion,} forms a lattice denoted $\op{Fil}(\mbf S)$.
Likewise, the map $s \mapsto \,\uparrow\! s$ embeds $\mbf S$ into $\op{Fil}(\mbf S)$.

We are now finally in the position to define our lattice $\op{PC}(\mbf S)$: it is the sublattice of $(\op{Id}\, \mbf S) \times (\op{Fil}\,\mbf S)$ generated by the diagonal (this idea is due to Day, cf.~\cite{RAD92}).

%%From here to the end of this section is background material stolen rather shamelessly from STA

\subsection{Constructing the finitely presented lattice $\op{FP}(\mbf S)$}\label{constructing_the_fin_presented_sec}

Recall that the construction of $\op{PC}(\mbf S)$ which we introduced above contains terms and subterms from both $R_+$ and $R_-$, and so in particular it depends on both $R_+$~and~$R_-$. We will now introduce a finitely presented lattice $\op{FP}(\mbf S)$. This lattice is constructed with $\op{PC}(\mbf S)$ as a starting point (and so it has the information from $R_-$ embedded in it); this lattice might be denoted more precisely as $\op{FP}(\op{PC}(\mbf S))$, but for ease of notation we follow the lighter notation $\op{FP}(\mbf S)$.

    \begin{defn}\label{def_S_presented} Let $\mbf S$ be a finite partially defined lattice. A lattice $\mbf F$ is the \emph{lattice finitely presented by} $\mbf S$ if there is an $\mbf S$-homomorphism (cf.~\ref{S_homo}) $\varphi : \mbf S \to F$ such that $F$ is generated by $\varphi(\mbf S)$ and $F$ satisfies the following mapping property:
if $\mbf L$ is a lattice and $\psi :\mbf S \to \mbf L$ is an $\mbf S$-homomorphism, then there is a lattice homomorphism
$f : \mbf F \to \mbf L$ such that $\psi = f \varphi$.
\end{defn}
Using this definition, it is easy to see that the lattice 
finitely presented by $\mbf S$ is unique up to isomorphism (see the next sentence for details).
This lattice is denoted by $\op{FP}(\mbf S)$. Regarding the existence of $\op{FP}(\mbf S)$, recall that $\mbf S$ arises from a sequence of variables $\mbf x = (x_1, ..., x_m)$, so recalling that $X = \{x_1, ..., x_m\}$, we have that $X$ is our intended generating set for $\mbf S$, and let $\mbf F(X)$ be the free lattice generated by $X$.
Let $\theta$ be the intersection of the kernels of all homomorphisms $\psi\sigma : \mbf F(X) \to \mbf L$ where $\sigma : \mbf F(X) \to \mbf S$ is the interpretation of $\mbf S$ as terms in $\mbf x$, and $\psi : \mbf S \to \mbf L$ is an $\mbf S$-homomorphism (once again, recall \ref{S_homo}), for some lattice $\mbf L$.
Then $\mbf F(X)/\theta$ is the lattice $\op{FP}(\mbf S)$ finitely presented by $\mbf S$ (in the sense of \ref{def_S_presented}).  

The \emph{word problem} for $\op{FP}(\mbf S)$ is: given
terms $u$ and $v$ with variables from $X$, to decide if the
interpretations of $u$ and $v$ in $\op{FP}(\mbf S)$ are equal.
Equivalently, is $(u,v) \in \theta$ (where $\theta$ is the congruence used to define $\op{FP}(\mbf S)$)?

Note that, so long as the presentation of $\mbf S$ is consistent, the non-inclusions $u_j \nleq v_j$ of $R_-$ play a minimal role.  
The terms $u_j(\mbf x)$ and $v_j(\mbf x)$ and their subterms are included in the construction of $\ol U$ in Skolem's algorithm, but \emph{consistency} means there is a model of $\la X,R_+ \ra$ in which $u_j \nleq v_j$.  
Put otherwise, Skolem's algorithm can be used to determine whether an arbitrary universal Horn sentence of the kind:
\begin{equation*}
(\ddagger) \qquad\quad
\&_{1 \leq i \leq k} \; s_i \leq t_i \implies u \leq v 
\end{equation*}
holds in lattice theory, by adding $u \nleq v$ to $R_-$ and testing it.  Thus Skolem's algorithm can be used to solve the word problem for $\op{FP}(\mbf S)$.

But, for better or worse, Skolem's 1920 solution was obscured, and this led to the solutions of Whitman \cite{PMW1941} for free lattices and Dean \cite{Dean1964} for finitely presented lattices, which in turn led to doubling constructions for $\mbf F(X)$ and $\op{FP}(\mbf S)$.

We now explain Dean's solution, as this will be relevant to our proof strategy.  Since $X \subseteq \mbf S$, there are evaluation maps:
\begin{align*}
\vare_1 :\, &\mbf F(X) \to \op{Idl}(\mbf S) \\
\vare_2 :\, &\mbf F(X) \to \op{Fil}(\mbf S)
\end{align*}
with $\vare_1 (x) = \,\downarrow\! x$ and $\vare_2 (x) = \,\uparrow\! x$ for $x \in X$.
Note that the evaluation $\vare_1 (w_1 \join \dots \join w_\ell) = \vare_1 (w_1) \join \dots \join \vare_1(w_\ell)$ which is the ideal generated by $\{ \vare_1 (w_1), \dots,\vare_1(w_\ell) \}$.
The dual statement holds for $\vare_2$, meets, and filters.

\begin{thm}[Dean's Theorem \cite{Dean1964}] \label{deansthm}
Let $u$ and $v$ be terms with variables in $X$. 
Then $u \leq v$ holds in $\op{FP}(\mbf S)$ if and only if one of the following holds:
\begin{enumerate}[(i)]
\item $u\in S$ and $v \in S$ and $u\leq v$ in $\mbf S$;
\item $u = u_1 \join \cdots \join u_k$ and $\forall\, i\; (u_i \leq v)$;
\item $v = v_1 \meet \cdots \meet v_k$ and $\forall\, j \; (u \leq v_j)$;
\item $u\in S$ and $v = v_1 \join \cdots \join v_k$ 
         and $u \in \vare_1 (v_1 \join \dots \join v_k) $;
\item $u = u_1 \meet \cdots \meet u_k$ and $v \in S$ 
        and $v \in \vare_2 (u_1 \meet \dots \meet u_k) $;
\item $u = u_1 \meet \cdots \meet u_k$ and $v = v_1 \join \cdots \join v_m$
      and $\exists \, i\; (u_i \leq v)$
      or $\exists \, j\; (u \leq v_j)$
      or $\exists \, s\in S \;u \leq s \leq v$. 
\end{enumerate}
\end{thm}

Item (vi) is a variation of Whitman's condition for free lattices, at this point we recall Alan Day's doubling construction as it is tightly related to Whitman's condition, as it will soon be clear.

Let $I$ be an interval of a lattice $\mbf L$ and let $\mbf L[I]$ be the disjoint union $(L-I) \,\dot\cup\, (I \times 2)$.
Order $\mbf L[I]$ by $x \leq y$ if one of the following holds:
\begin{enumerate}[(1)]
\item $x$, $y \in \mbf L-I$ and $x \leq y$ holds in~$\mbf L$;
\item $x$, $y \in I \times2$ and $x\leq y$ holds in~$I \times 2$;
\item $x \in \mbf L- I$, $y= (u,j) \in I \times2$, and $x \leq u$ holds in $\mbf L$; 
\item $x= (v,j) \in I \times 2$, $y \in \mbf \mbf L-I$, and $v\leq y$ holds in $\mbf L$.
\end{enumerate}
There is a natural map~$\lambda$ from $\mbf L[I]$ back onto $\mbf L$ given by
\begin{equation*} 
\lambda(x) =\begin{cases} x&\text{if $x \in \mbf L-I$},\\
                   v&\text{if $x=(v,j)\in I\times2$}.
\end{cases}
\end{equation*}

\begin{thm}[{\cite{RAD77}\label{doubling}}]
Let $I$ be an interval in a lattice $\mbf L$. Then
$\mbf L[I]$ is a lattice and $\lambda : \mbf L[I] \to \mbf L$ is a lattice epimorphism.
\end{thm}

We will revisit doubling in more detail in Subsection~\ref{doubling_lemmas}.

Now Whitman \cite{PMW1941} showed free lattices satisfy the condition:
\begin{align*}
(\op W) \quad \text{if }u = u_1 \meet \cdots \meet u_k \leq v = &\; v_1 \join \cdots \join v_m  \\
\text{then } \exists\, i\in [1, k]\; u_i \leq v &\text{ or } \exists\, j \in [1, m]\; u \leq v_j .     
\end{align*} 
Item (vi) of Dean's Theorem is a variation of (W) for $\op{FP}(\mbf S)$:
\begin{align*}
(\op{WS}) \quad \text{if }u = u_1 \meet \cdots \meet u_k \leq &\; v = v_1 \join \cdots \join v_m  \\
\text{then } \exists\, i \in [1, k] \; u_i \leq v &\text{ or } \exists\, j \in [1, m] \; u \leq v_j \text{ or }\exists\, s\in S \;u \leq s \leq v.   
\end{align*} 
so that (W) is (WS) for the case when $R_+ \,\cup\, R_- = \varnothing$.

\begin{notation}\label{W_failures}
When (W) (resp.~(WS)) fails, we call $[u,v]$ a W- (resp.~WS-) \emph{failure interval}.
A (W)-failure interval $I = [u,v]$ is \emph{$\mbf S$-disjoint} if\/ $I \cap \mbf S = \varnothing$. We now explain how (W)- (resp.~(WS)-) failures are related to the doubling construction.
\end{notation} 
%% stolen from STA
A lattice $\widehat{\mbf L}$ is the \emph{$\mathrm{W}$-cover} of a lattice $\mbf L$ if $\widehat{\mbf L}$ satisfies (W),
and there is a epimorphism $f \colon \widehat{\mbf L} \onto \mbf L$ such that
if $\mbf K$ is a lattice satisfying (W) with a epimorphism 
$g$ onto $\mbf L$, then there is a epimorphism $h \colon \mbf K \onto \hat{\mbf L}$ such that $g = f \circ h$. 
The reader can verify that the $\mathrm{W}$-cover of a lattice is unique and always exists. 
The $\mathrm{WS}$-cover of a lattice containing $\mbf S$ is defined analogously.

If a lattice satisfies (W) then $\widehat{\mbf L} = \mbf L$, of course. 
On the other hand the $\mathrm{W}$-cover of the free distributive lattice on $n$ generators is $\mbf F_n$, as was shown by Alan Day \cite{RAD77}.
In fact he showed $\mbf F_n$ could be constructed from the free distributive lattice on $n$ generators using his doubling construction repeatedly to fix (W)-failures.
This played an important role in his proof that finitely 
generated free lattices are weakly atomic~\cite{RAD77},
i.e., if $u > v$ then there exist $s$, $t$ such that $u \geq s \succ t \geq v$. 

Of course a finitely presented lattice may fail (W). 
Indeed, every finite lattice is finitely presented. Similarly a finitely presented lattice containg $\mbf S$ may fail (WS). But as a W-failure in a lattice can be corrected by a doubling construction, similarly, a WS-failure in a lattice $\mbf L$ containing $\mbf S$ can likewise be corrected by the doubling construction, though this may create new WS-failures (as for the case of W-failures).  
%Rather, we have seen by Dean's Theorem (i.e., \ref{deansthm}) that the (W)-failures in $\op{FP}(\mbf S)$ are precisely the (WS)-failures.
What Day showed in \cite{RAD92} is essentially that $\op{FP}(\mbf S)$ can be constructed from $\op{PC}(\mbf S)$ by repeatedly using the doubling construction to fix (WS)-failures and taking the inverse limit. 

We end our section with a recap of our construction.
The second bullet in the next theorem represents a slight technical adjustment, which will be used in Section~\ref{subsec:part2}.
It allows us to test whether $h(u) \leq h(v)$, when $u \nleq v$ is in $R_-$ and $h: \mbf S \to \mbf L$ is a PDL homomorphism, by checking whether $(u \join v,v)$ is in the equivalence relation $\ker h$.  We need $u \join v$ to be defined in $\mbf S$ for this to happen.

\begin{thm}[Our Construction]\label{thm:daycons}
How to make a finitely presented lattice. 
\begin{enumerate}[(1)]
\item[$\bullet$]  Given is a finite presentation $\la \mathbf{x}, R_+, R_- \ra$.
\item[$\bullet$]  Let $U$ consist of all subterms of $\{ s, t \}$ with $s \leq t$ in $R_+$, and all subterms of $u \join v$ whenever $u \nleq v$ is in $R_-$.
\item[$\bullet$] Convert $U$ to a \emph{partially defined lattice} $\mbf S$ a la Skolem (cf. Theorem~\ref{skolem} and the discuss immediately following it).
\item[$\bullet$] Form the ideal lattice and filter lattice of\/ $\mbf S$, ordered by inclusion and reverse inclusion, respectively.
\item[$\bullet$] Form the partial completion $\op{PC}(\mbf S)$, the sublattice of $(\op{Id}\, \mbf S) \times (\op{Fil}\,\mbf S)$ generated by the diagonal. 
\item[$\bullet$] Iterate:  $\mbf P_0 = \op{PC}(\mbf S)$, and $\mbf P_{k+1}$ fixes the WS-failures in $\mbf P_k$ (equiv. $\mbf P_{k+1}$ fixes the $\mbf S$-disjoint $\mathrm{W}$-failures in $\mbf P_k$). 
        Note there is a homomorphism $\mbf P_{k+1} \onto \mbf P_k$.
\item[$\bullet$] Then $\op{FP}(\mbf S)$ is the sublattice generated by $\mbf S$ of the inverse limit of this sequence.   
\end{enumerate}
\end{thm}

In other words, $\op{FP}(\mbf S)$ is the WS-cover of $\op{PC}(\mbf S)$.
The details of this construction are not hard and are presented in \cite{RAD77,RAD_it, RAD92, DGP} and in  Section II.7 of~\cite{FJN}; see also Section~2-8 of~\cite{FN_STA}.

\subsection{Standardization}
For technical reasons, before implementing the strategy, we need to \emph{standardize} the $\mathrm{PDL}$ $\mbf S$ and the negations in $R_-$. 
Remember that the elements of $\mbf S$ are equivalence classes of terms that evaluate the same in the finitely presented lattice $\op{FP}(\mbf S)$ where $\mbf S$ is determined by $\mbf x$ and $R_+$.
Each element of $\mbf S$ can be represented (not necessarily uniquely) by a free lattice terms that are either a formal meet, or a formal join, or both, or neither (a formal generator).  

%\noindent
Recall that we start with a presentation $\la \mbf x,R_+,R_- \ra$ and by Skolem's algorithm produce a $\mathrm{PDL}$ $\mbf S$.  Let $S_+$ denote the positive relations $s(\mbf x) \leq t(\mbf x)$ that hold in $\mbf S$.  Of course $R_+ \subseteq S_+$, but the latter also contains relations that are a consequence of those in $R_+$.
For example, if $t = \Join t_j$ then $S_+$ contains $t_j \leq t$ for all $j$.  But others could be less trivial; that is why we have Skolem's algorithm.  Observe that:
\begin{quotation}
   $\la \mbf x,R_+,R_- \ra$ occurs in a lattice $\mbf L$ if and only if $\la \mbf x,S_+,R_- \ra$ occurs in $\mbf L$.  
\end{quotation}

\begin{defn}\label{def_standard} We say that a presentation $\la \mbf x, R_+, R_- \ra$ with associated $\mathrm{PDL}$ $\mbf S$ is standardized if it satisfies:
\begin{enumerate}[(1)]
    %\item no element of $S$ is both a join and a meet, that is, $s = \Meet p_i$ and $s = \Join q_j$ cannot both be in $S_+$;   
    \item  $\mbf S$ satisfies $(\op W)$, that is, there is no relation of the form $\Meet s_i \leq  \Join t_j$ in $S_+$
    unless one of the inclusions $s_i \leq t$ or $s \leq t_j$ is also in $S_+$;
    \item there is no relation of the form $\Join u_i \nleq v$ in $R_-$;
    \item there is no relation of the form $u \nleq \Meet v_j$ in $R_-$.    
\end{enumerate}
\end{defn}

\begin{lm} \label{lm:std} For each presentation $\la \mbf x, R_+, R_- \ra$ there is a finite collection of standardized presentations $\la \mbf x, T_1,U_1 \ra, \dots, \la \mbf x, T_m,U_m \ra$ such that $\la \mbf x, R_+, R_- \ra$ occurs in a lattice $\mbf L$ satisfying $(\op W)$ 
%free lattice $ \mbf F$ 
if and only if some $\la \mbf x,T_j,U_j \ra$ occurs in $\mbf L$.   %% or $\mbf F$.
\end{lm}

\begin{proof}
We describe a process of replacing a given presentation $\mbf S = \la \mbf x, R_+, R_- \ra$ by a collection of standardized ones. The process shows how to reduce the number of failures of (1)--(3) (cf.~\ref{def_standard}) in $\la \mbf x, S_+, R_- \ra$; at each step we reduce the number of instances of failures of (1)--(3), and this will eventually lead to a finite collection of standardize presentations, as $\la \mbf x, R_+, R_- \ra$ is a finite object.

\smallskip \noindent 
%(1)  Assume some $s \in S$ is both, so that $s = \Meet p_i$ and $s = \Join q_j$ are among the defining relations of $R_+$.  Let $\mbf P_i$ be the PDL corresponding to $\la \mbf x, R_+ \cup \{s = p_i \}, R_- \ra$ and likewise let $\mbf Q_j$ be the PDL corresponding to  $\la \mbf x, R_+ \cup \{s = q_j \}, R_- \ra$.
%  Then $\la \mbf S,R_- \ra$ occurs in a free lattice (or indeed any lattice satisfying (W)) if and only if some $\la \mbf P_i, R_- \ra$ occurs or some $\la \mbf Q_j,R_- \ra$ occurs.

\smallskip \noindent (1) Assume there is a relation $\Meet s_i \leq  \Join t_j$ in $R_+$.
  Set $S_i = R_+  \cup \{( s_i,t) \}$ and $T_j = R_+  \cup \{( s,t_j) \}$.  
  Then $\la \mbf x, R_+,R_- \ra$ occurs in a lattice satisfying (W) if and only if some $\la \mbf x, S_i, R_- \ra$ occurs or some $\la \mbf x,T_j,R_- \ra$ occurs.

\smallskip \noindent (2) If some pair $(u,v)$ is in $R_-$ with $u = \Join u_i$ a formal join, let $U_i = R_- \setminus \{(u,v) \} \cup \{ (u_i,v) \}$.  Then $\la \mbf x,R_+,R_- \ra$ occurs in a lattice if and only if $\la \mbf x, R_+, U_i \ra$ occurs for some $i$. 

\smallskip \noindent (3) Dually, if there is a pair $(u,v)$ in $R_-$ with $v = \Meet v_j$ a formal meet, let $V_j = R_- \setminus \{(u,v) \} \cup \{ (u,v_j) \}$.  Again $\la \mbf x, R_+,R_- \ra$ occurs in a lattice if and only if $\la \mbf z, R_+, V_i \ra$ occurs for some $i$. 
\end{proof}

After making these adjustments, thereby replacing the original problem by a disjunction of standardized problems, we can assume that the elements of $\mbf S$ are formally generators, or meets, or joins, but not both a meet and a join.  Moreover, we can assume that the non-inclusions of $R_-$ are of the form (gen,gen), or (meet, gen), or (gen, join).  
%\gianluca{The referee says (referring to the previous sentence): Actually the elements of S are equivalence classes of terms. What is meant by the statements is certainly true but some care in the wording is necessary.}
%I changed this at the beginning of 2.3

Then apply the Skolem algorithm to each $\la T_j,U_j \ra$ to see if the presentation is consistent with lattice theory. For those that are, form the partial completion $\op{PC}(\mbf T_j)$, and continue as below to see if it can be realized in a free lattice.  

\subsection{Standardization and doubling} \label{doubling_lemmas}

Recall that there is a natural homomorphism $\lambda : L[I] \to L$.  We are interested in the following situation: given a $\mathrm{PDL}$ $\mbf S$ and a $\mathrm{PDL}$ homomorphism $r_0 : \mbf S \to L$, when can we find a $\mathrm{PDL}$ homomorphism $r_1 : \mbf S \to L[I]$ such that $r_0 = \lambda r_1$?  That means $r_1(s)$ will be $r_0(s)$ when $r_0(s) \notin I$, and either $(r_0(s),0)$ or $(r_0(s),1)$ when $r_0(s) \in I$.  
If it is always possible to find such an $r_1$, for any $r_0$, $L$, and $I$, then $\mbf S$ is said to be \emph{stable for doubling}.  

\begin{lm} \label{stabledoub}
If\/ $\la \mbf x,R_+,R_- \ra$ is standardized (cf.~\ref{def_standard})
and\/ $\mbf S$ is the associated $\mathrm{PDL}$,  then $\mbf S$ is stable for doubling.
\end{lm}
\begin{proof}
Let $I=[a,b]$.
If $r_0(s) \notin I$, then we must assign $r_1(s)=r_0(s)$.  Thus we consider an elements $s \in \mbf S$ with $r_0(s) \in I$. 
% Recall that the relations of $R_+$ are of the form (gen) $\leq$ (gen), or
%(meet) $\leq$ (gen), or (gen) $\leq$ (join).

\smallskip \noindent Form the set $GM$ recursively as follows:
\begin{itemize}
 \item $GM_0 = \varnothing$
 \item $s \in GM_{k+1}$ if $s \in \op{GM}_k$, or if $r_0(s) \in I$ and $s \geq \Meet u_i$ in $R_+$, where for all $i$, either $a \leq r_0(u_i) \nleq b$, or $u_i \in GM_k$.  
\end{itemize}
Let $GM = \bigcup_k GM_k$.  
Note that $s \leq t$ and $s \in GM$ implies $t \in GM$.
%Finally, let $GM'$ be the filter on $\{ s: r_0(s) \in I \}$ generated by $GM$: if $s_k \in GM$ for all $k$ and $\Meet s_k$ is defined in $S_+$ and $\Meet s_k \leq p$, let $p \in GM'$. 
Then $r_1(s)$ is forced to be $(r_0(s),1)$ if and only if $s \in GM$.  

\smallskip \noindent Form $LJ$ dually:
\begin{itemize}
 \item $LJ_0 = \varnothing$
 \item $s \in LJ_{k+1}$ is $s \in \op{LJ}_k$, or 
 if $r_0(s) \in I$ and $s \leq \Join v_j$ in $R_+$, where for all $j$, either $a \nleq r_0(v_j) \leq b$, or $v_j \in LJ_k$.  
\end{itemize}
Let $LJ = \bigcup_k LJ_k$. 
Note that $w \leq s \in LJ$ implies $w \in LJ$.
%Again for $LJ'$ by closing under defined joins in $R_+$ and upward.
Then $r_1(s)$ is forced to be $(r_0(s),0)$ if and only if $s \in LJ$.  

\smallskip \noindent Because $\mbf S$ is standardized,
$GM \cap LJ = \varnothing$.  
Indeed, suppose $z \in GM \cap LJ$.
Then $z \in GM_k \cap LJ_\ell$ for some minimal $(k,\ell)$.  Use double induction on $(k,\ell)$, with $k=0$ or $\ell = 0$ both being impossible.  
Thus $\Meet u_i \leq s \leq \Join v_j$ where for each $i$, either 
either $a \leq r_0(u_i) \nleq b$, or $u_i \in GM_{k-1}$, and for each $j$ either $a \nleq r_0(v_j) \leq b$, or $v_j \in LJ_k$.
Since $\la \mbf x,R_+,R_- \ra$ is standardized, $\mbf S$ satisfies $(\op W)$.  So either $u_i \leq \Join v_j$ for some $i$, or $\Meet u_i \leq v_j$ for some $j$.
If $u_i \leq \Join v_j$ and $u_i \in GM_{k-1}$, then $u_i \in GM_{k-1} \cap LH_\ell$, contradicting the inductive hypothesis.
If $u_i \leq \Join v_j$ and $r_0 (u_i) \nleq b$, then $r_0 (\Join v_j) \nleq b$, contradicting the assumption $r_0(\Join v_j) \in I$.  
The dual cases are handled similarly.  We conclude that $GM \cap LJ = \varnothing$. 

\smallskip \noindent Thus, we can assign:
\[
r_1(s) = \begin{cases} r_0(s) &\text{ if } r_0(s) \notin I, \\
(r_0(s),1) &\text{ if } s \in GM, \\ (r_0(s),0) &\text{ otherwise. }
%% the last case covers s \leq t for other elements, preserving the order 
\end{cases} \]

\end{proof}
%%correctable by doubling removed here

\subsection{Bounded and projective lattices}\label{bounded_sec}

In this section, we review the background and characterization of finitely generated projective lattices. This will be relevant for our ``Basic Lemma'', i.e., Lemma~\ref{lm:basic}.

A \emph{join cover} of an element $p \in \mbf L$ is a finite subset $A \subseteq \mbf L$ such that $p \leq \Join A$.
The join cover is \emph{nontrivial} if $p \nleq a$ for all $a \in A$.
An element $p$ in a lattice $\bf L$ is \emph{join prime} if it has no nontrivial join cover, i.e., $p \leq \Join A$ implies $p \leq a$ for some $a \in A$.  
For finite subsets $A$, $B \subseteq \mbf L$, we say that $A$ \emph{refines} $B$, written $A \ll B$, if for every $a \in A$ there exists $b \in B$ such that $a \leq b$.  
The join cover $p \leq \Join A$ is \emph{minimal} if whenever $p \leq \Join B$ and $B \ll A$, then $A \subseteq B$.
The dual notions, including \emph{meet cover} and \emph{meet prime}, are defined analogously.

 Let $\mbf K$ and $\mbf L$ be lattices. 
 A homomorphism $f: \mbf K \rightarrow \mbf L$ is said to be \emph{lower bounded} if for every $a \in \mbf L$, the set $\{ u \in \mbf K: f(u) \geq a \}$ is either empty or has a least element. 
 A finitely generated lattice $\mbf L$ is called \emph{lower bounded} if every homomorphism $f: \mbf K \rightarrow \mbf L$, where $\mbf K$ is finitely generated, is lower bounded. 
 Let $D_0(\mbf L)$ denote the set of join prime elements of\/ $\mbf L$. % i.e., those elements which have no nontrivial join-cover (cf.~\cite[pg. 29]{FJN}). 
For $k > 0$, let $a \in D_k (\mbf L)$ if every nontrivial join-cover $V$ of $a$ has a refinement $U \subseteq D_{k-1}(\mbf L)$ which is also a join-cover of $a$. 
Observe that, from the definition, $D_0(\mbf L) \subseteq D_1(\mbf L) \subseteq D_2(\mbf L) \subseteq \cdots$.
Let $D(\mbf L) = \bigcup_{k < \omega} D_{k}(\mbf L)$.

The characterization of (finitely generated) lower bounded lattices is due to Kostinsky \cite{Kostinsky}, building on J\'onsson \cite{JN77} and McKenzie \cite{McK1972}.
See \cite[Theorem~2.13]{FJN}.

\begin{thm} \label{thm:kos}
For a finitely generated lattice $\mbf L$, the following are equivalent.
\begin{enumerate}[(1)]
    \item There exists a finite set $X$ and a lower bounded epimorphism $f : \mbf F(X) \onto \mbf L$.
    \item For every finitely generated lattice $\mbf K$, every homomorphism $h : \mbf K \to \mbf L$ is lower bounded.
    \item $D(\mbf L)=\mbf L$.
\end{enumerate}
\end{thm}

The notions of upper bounded homomorphism, upper bounded lattice, etc.~are defined dually.  A lattice is \emph{bounded} if it is both lower and upper bounded.  Day characterized finite bounded lattices in terms if doubling intervals \cite{RAD77}.

\begin{thm} \label{thm:RAD1}
A finite lattice is bounded if and only if it can be obtained from a 1-element lattice by a sequence of doubling intervals.
\end{thm}

Since every finite distributive lattice can be obtained from a 1-element lattice by a sequence of doublings, we could also start with a finite distributive lattice.
There are analogues for finite lower and upper bounded lattices in terms of doubling other types of convex sets \cite{RAD79}.

A lattice $\mbf L$ is said to be \emph{projective} if for any epimorphism $f : \mbf M \twoheadrightarrow \mbf N$ and any homomorphism $h : \mbf L \to \mbf N$ there exists a homomorphism $g : \mbf L \to \mbf M$ such that $h=fg$.  

A lattice $\mbf L$ is a \emph{retract} of $\mbf K$ if there exist homomorphisms $f : \mbf K \onto \mbf L$ and $g : \mbf L \to \mbf K$ with $fg=\mathrm{id}_{\mbf L}$.  Thus $g(\mbf K)$ is a transversal of $f$ and $g$ an embedding.

\begin{thm}[Theorem 5.1 in \cite{FJN}] \label{thm:retract} 
The following are equivalent for a lattice $\mbf L$.
\begin{enumerate}[(1)]
\item $\mbf L$ is a projective lattice.
\item $\mbf L$ is a retract of a free lattice.
\item For any $\mbf K$, any surjective $f: \mbf K \onto \mbf L$  has a retraction.
\end{enumerate}
\end{thm}

Combining this with Corollaries 5.9 and 5.10 in \cite{FJN}, due to Kostinsky \cite{Kostinsky}, we get:
\begin{thm} \label{thm:proj}
The following are equivalent for a finitely generated lattice.
\begin{enumerate}[(1)]
\item $\mbf L$ is a projective lattice.
\item $\mbf L$ is a sublattice of a free lattice.
\item $\mbf L$ satisfies $\op D(\mbf L) = \mbf L = \op D^d(\mbf L)$ and $(\mathrm W)$.
\item $\mbf L$ is bounded and satisfies $(\mathrm W)$.
\end{enumerate}    
\end{thm}
The equivalence of (3) and (4) uses Theorem~\ref{thm:kos} and its dual.
The non-finitely generated case is in the rest of Chapter~5 of \cite{FJN} and another paper \cite{FN1978}. As a consequence, we have the following result, i.e., Theorem~\ref{covers_are_proj}. Before stating it we need a couple of definitions which will be used in the proof of \ref{covers_are_proj}. Recall that an element $a$ in a lattice $\mbf L$ is said to be join irreducible if $a = b \vee c$ implies that either $a = b$ or $a = c$. Notice also that being join irreducible is a first-order property. We denote by $\text{J}(\mathbf{K})$ the set of join irreducible elements in $\mathbf{K}$.

 \begin{defn}\label{def_join_dep_rel} Let $\mbf{L}$ be a lattice and $u, v \in \mathbf{L}$. We define $u D v$ if $u \neq v$, $v$ is join irreducible, and there exists $q$ such that $u \leq v \vee q$ but $u \not\leq s \vee q$ for every $s < v$.
\end{defn} 

\begin{thm}\label{covers_are_proj} A finite lattice $\mbf L$ is bounded if and only if $\mbf L[I]$ is bounded, for every interval $I$ of $\mbf L$. In particular, a finite lattice $\mbf L$ is bounded if and only if its $\op W$-cover is a projective lattice.  
\end{thm}

\begin{proof}
One direction, that the W-cover of a finite bounded lattice is projective, is a theorem in Day~\cite{RAD92} (middle of page 262).
The heuristic is as follows.
When you double an interval, it creates a new join irreducible element, say $p_1$.  But the minimal join covers of existing join irreducibles are unchanged.  
If $x \in \op{J}(\mbf L)$ is in $\op D_k(\mbf L)$, 
then $x$ (or if doubled $x_0$) is in $\op D_k(\mbf L[I])$.
Thus there are no new $\op D$-relations $x \,\op D\, p_1$, only the other way around, $p_1 \,\op D\, x$.  
If $\op J(\mbf L) \subseteq \op D_k(\mbf L)$,
then $\op J(\mbf L[I]) \subseteq \op D_{k+1}(\mbf L[I])$.
This is preserved when you take the limit.
%So you cannot create new $\op D$ cycles or ``ascending'' $\mbf D$-sequences . 

\smallskip \noindent One the other hand, if a finite lattice $\mbf L$ is not bounded, then it contains a $\op D$-cycle.  By the same reasoning as above, since minimal join covers of existing join irreducibles are unchanged, the same cycle is in $\mbf L[I]$ and in the limit.
\end{proof}

\subsection{The basic lemma}%%the lemma and retracts

Recall that a lattice presentation with negations has three components:  the generating set $X = \{x_1, ..., x_m \} = \mbf{x}$, the inclusions $R_+$, and non-inclusions $R_-$.  
If $\mbf L$ is a lattice (or partial lattice) generated by $X$ and containing $\mbf S$, then a homomorphism $h: \mbf L \to \mbf K$ is an $\mbf S$-homomorphism if $h(s) \leq h(t)$ whenever $s \leq t$ is in $R_+$, and $h(u) \nleq h(v)$ whenever $u \nleq v$ is in $R_-$.
We say that \emph{$\mbf S$ occurs in $\mbf K$} if there is an $\mbf S$-homomorphism $h : \mbf S \to \mbf K$. 
In this case $h$ is a partial lattice homomorphism.

Now assume $\mbf L$ is a lattice generated by $X$, $\mbf G$ is a lattice, and $h : \mbf L \onto \mbf G$ is a surjective $\mbf S$-homomorphism.
If there is a retraction $\rho : \mbf G \to \mbf L$ with $h\rho = \mathrm{id}_{\mbf G}$, then $\rho$ is an embedding.
Hence, $r = \rho h$ is an $\mbf S$-homomorphism with $r : \mbf L \to \mbf L$ and $r^2=r$.
Let us call such an $r$ an \emph{$\mbf S$-retraction} of $\mbf L$. Here is our Basic Lemma.

\begin{lm}[The Basic Lemma]\label{lm:basic}
Let $\mbf S$ be a finite partially defined lattice. Then the following are equivalent:
\begin{enumerate}[(1)]
\item  $\mbf S$ occurs in $\mbf F(Y)$ for some $Y$;
\item  there is an $\mbf S$-homomorphism $h: \op{FP}(\mbf S) \to \mbf F(Y)$ for some $Y$;
\item  there is a projective lattice $\mbf G$ and an $\mbf S$-retraction $r : \op{FP}(\mbf S) \onto \mbf G$;
\item $\op{FP}(\mbf S)$ contains an $\mbf S$-retract that is bounded and satisfies $(\mathrm W)$;
\item $\op{FP}(\mbf S)$ contains an $\mbf S$-retract s.t. $\op D(\mbf L) = \mbf L = \op D^d(\mbf L)$ and $(\mathrm W)$.
\end{enumerate}
\end{lm}

\begin{proof}
The equivalence of (1) and (2) is the definition of \emph{finitely presented lattice}, and the equivalence of (3)--(5) is Theorem~\ref{thm:proj}.

\smallskip \noindent For (2) $\to$ (3), assume that the homomorphism $h$ of (2) exists, and take $\mbf G_0$ to be the image $h(\op{FP}(\mbf S))$ in $\mbf F(Y)$.
Then $\mbf G_0$ is generated by $h(\mbf S)$ (or rather by $X = \{x_1, ..., x_m \} = \mbf{x}$), so we can apply Theorem~\ref{thm:retract}:  a finitely generated lattice is projective if and only if it is a sublattice of a free lattice.
Thus there is a retraction $\rho : \mbf G_0 \to \op{FP}(\mbf S)$ with $h \rho = \mathrm{id}_{\mbf G_0}$.  For (3), we put $\mbf G = \rho(\mbf G_0) = \rho h (\op{FP}(\mbf S))$.

\smallskip \noindent Now assume (3) holds, and again using Theorem~\ref{thm:proj}, then there exists $Y$ and an embedding $\vare : \mbf G \rightarrow \mbf F(Y)$.  For (2), let $h = \vare r$.
\end{proof}

    This ends our preliminary section, toward a proof of decidability of the universal theory of free lattices, as by \ref{lm:basic} we have that there is an $\mbf S$-homomorphism $h : \op{FP}(\mbf S) \to \mbf F(Y)$ if and only if there is an $\mbf S$-retract
$r : \op{FP}(\mbf S) \onto \mbf G$ with $\mbf G$ projective, so we are left to understand when the latter condition holds, and in particular to show that this can be checked algorithmically, this is the content of the next section.

\section{Implementing the strategy}

\subsection{Reduction to finite bounded lattices} \label{subsec:reduction}
Now, let us put everything together in order to determine when a standardized presentation occurs in a free lattice.  The first step, in this subsection, is to reduce the problem to determining when it occurs in a finite bounded lattice. 
The second step, in the next subsection, is to show how to test that.  While reading the next proof we invite the reader to refer to Figure~\ref{fig:maps}.

\begin{thm} \label{thm:reduction}
Assume that $\la \mbf S, R_- \ra$ is standardized (cf.~\ref{def_standard}) and consistent.
Then $\la \mbf S,R_-\ra$ occurs in a free lattice if and only if\/ $\la \mbf S,R_-\ra$ occurs in a finite bounded lattice.
\end{thm}

\begin{proof}
Assume that $\la \mbf S,R_- \ra$ is modeled in a free lattice.  Then by Lemma~\ref{lm:basic} there is a retraction $\rho$ on $\op{FP}(\mbf S)$ such that $\rho$ respects $R_-$, that is, $\rho(u) \nleq \rho(v)$ for all $(u,v) \in R_-$, and the image $\mbf G = \rho(\op{FP}(\mbf S))$ is a projective lattice.

\smallskip \noindent 
Recall Alan Day's construction of the finitely presented lattice $\op{FP}(\mbf S)$ from the partial completion $\op{PC}(\mbf S)$, which is embedded in our Theorem~\ref{thm:daycons}.
Let $\mbf P_0 = \op{PC}(\mbf S)$.  Recursively form $\mbf P_{k+1}=\mbf P_k[I_k]$ as an interval doubling that fixes a WS-failure in $\mbf P_k$ (i.e., an $\mbf S$-disjoint $\mathrm{W}$-failure (recall \ref{W_failures}) in $\mbf P_k$). There is a homomorphism $\lambda_k :\mbf P_{k+1} \onto \mbf P_k$.
Then $\op{FP}(\mbf S)$ is the sublattice generated by $\mbf S$ of the inverse limit of this sequence, and there are natural projections $\pi_k : \op{FP}(\mbf S) \onto \mbf P_k$.
 
\smallskip \noindent 
Now the projective retract $\mbf G$ of the first paragraph is a sublattice of $\op{FP}(\mbf S)$, so we can restrict the maps $\pi_k$ to $\mbf G$.  Let $\mbf Q_k = \pi_k(\mbf G)$, 
and set $\mbf H = \pi_0(\mbf G)$.
Since $\mbf Q_k = \pi_k \rho(\op{FP}(\mbf S))$, each $\mbf Q_k$ is a model of $S_+$.  
In fact, the elements of $\mbf S$ have a unique pre-image under $\pi_0$, and the map $r_k = \pi_k \rho \pi_k^{-1}$ is a $\mathrm{PDL}$ homomorphism from $\mbf S$ to $\mbf Q_k$.   
%That is, it preserves the inclusions and operations of $R_+$ (though maybe no more).  

\smallskip 
\noindent 
On the other hand, notice that if $(u,v) \in R_-$, then $u$ and $v$ are in $\mbf S$.  Now $\rho(u) \nleq \rho(v)$ in $\mbf G$ as $\mbf G$ models $R_-$.  
%It might be that $\mbf H \models r(u) \leq r(v)$, but 
Thus for some $k_0$ in the sequence of projections we have $r_{k_0}(u) \nleq r_{k_0}(v)$.
(In general, $r_{k_0}\pi_{k_0} = \pi_{k_0} \rho$, but $u$, $v$ are in $S$ and each $\pi_j$ is the identity on $S$.)

\smallskip 
\noindent 
It remains to observe that each $\mbf Q_k$ is a finite bounded lattice.  
In fact, just as $\op{FP}(\mbf S)$ is obtained from $\op{PC}(\mbf S)$ by a sequence of doubling intervals (Theorem~\ref{thm:daycons}), so $\mbf G$ is obtained from $\mbf H$ by restricting those doublings to the appropriate subsets.  But $\mbf G$, being projective, actually satisfies (W), and is generated by the image of $\mbf S$, so $\mbf G$ is the W-cover of $\mbf H$.   
By Theorem~\ref{covers_are_proj}, the lattice $\mbf Q_0 =\mbf H$ is bounded.  As in the proof of that theorem, for a finite lattice $\mbf L$ we have that $\mbf L[I]$ is bounded if and only if $\mbf L$ is bounded (Day \cite{RAD79}).  %%could be RAD77
Since $\mbf Q_{j+1}$ is obtained by doubling an interval in $\mbf Q_j$, inductively each $\mbf Q_k$ is bounded. 

\smallskip 
\noindent Thus $\mbf Q_{k_0}$ witnesses that $\la \mbf S,R_-\ra$ occurs in a finite bounded lattice.

\begin{figure}
\begin{center}
\begin{tikzpicture}[scale=1]
%\tikzstyle{every node}=[draw,circle,fill=white,minimum size=5pt,inner sep=0pt,label distance=1mm] 
   \node at (0,-.1)[label=above: $\op{PC}(\mbf S) \supseteq \mbf S$]{};
   \node at (3,0)[label=above: $\mbf H$]{};
   \node at (0-.3,2-.1)[label=above: $\mbf P_k$]{};
   \node at (3,2-.1)[label=above: $\ \mbf Q_k$]{};
   \node at (0-.3,4-.1)[label=above: $\op{FP}(\mbf S)$]{};
   \node at (3,4)[label=above: $\mbf G$]{};

\node at (1.5,4.2)[label=above: $\rho$]{};
\node at (1.5,2.2)[label=above: $r_k$]{};
\node at (1.5+.2,0.2)[label=above: $r_0$]{};

\node at (0-.15,1.3)[label=left: $\lambda_0 \cdots \lambda_{k-1}$]{};
\node at (3-.1,1.3)[label=right: $\lambda_0 \cdots \lambda_{k-1}$]{};
\node at (0-.15,3.3)[label=left: $\pi_k$]{};
\node at (3-.1,3.3)[label=right: $\pi_k$]{};
     
   \draw [->] (1.1,.3) -- (2.65,.3);
  \draw [->] (.7-.3,2.3) -- (2.65,2.3);
  \draw [->] (.7-.3,4.3) -- (2.65,4.3);
  
  \draw [->] (0-.2,1.9) -- (0-.2,.7);
  \draw [->] (3,1.9) -- (3,.7);
  \draw [->] (0-.2,3.9) -- (0-.2,2.7);
  \draw [->] (3,3.9) -- (3,2.7);
  \end{tikzpicture}
\end{center}
\caption
%\textcolor{orange}{Copied}
{Maps used in the proof of Theorem \ref{thm:reduction}} \label{fig:maps} 
\end{figure}

\bigskip

%%sufficiency
\noindent 
Conversely, assume there exist a finite bounded lattice $\mbf K$
 and a $\mathrm{PDL}$ homomorphism $r: \mbf S \to \mbf K$ that respects $R_-$,
 that is, $r(u) \nleq r(v)$ for all pairs $(u,v) \in R_-$.
 Form the W-cover $\widehat {\mbf K}$ as the sublattice generated by $\mbf S$ of the limit of a sequence where $\mbf M_0 = \mbf K$ and $\mbf M_{j+1} = \mbf M_j[I_j]$ with $I_j$ a W-failure interval in $\mbf  M_j$.   
 Then $\widehat{\mbf K}$ is projective by Theorem~\ref{covers_are_proj}.
 Moreover, each $\mbf M_j$ is a model of $\la \mbf S, R_- \ra$ by Lemma~\ref{stabledoub}, using the assumption that the presentation is standardized.  Hence the limit $\mbf M_\omega = \widehat{\mbf K}$ is also a model, and thus a projective model.  %as desired.
\end{proof}

%%%%%% to here %%%%%%%

%%%%%%%  part II to be inserted below %%%%%%%

\subsection{Testing whether $\la \mbf S, R_- \ra$ occurs in a finite bounded lattice} \label{subsec:part2}
Theorem~\ref{thm:reduction} tells us that $\la \mbf S, R_- \ra$ occurs in a free lattice if and only if it occurs
in a finite bounded lattice $\mbf K$, but it does not tell us enough about $\mbf K$ to be turned into
an effective test.  This section provides an algorithm to decide that.

A \emph{pseudovariety} is a class of finite algebras closed under homomorphic images, subalgebras, and finite direct products.
Lattice pseudovarieties include all finite lattices in the following list:
\begin{itemize}
\item  $\msc D$, distributive lattices;
\item $\msc V$, any lattice variety;
\item $\msc {SD}$, semidistributive lattices;
\item $\msc B$, bounded lattices.
\end{itemize}
 For any pseudovariety $\msc K$ and any finite lattice $\mbf L$, there is a least congruence $\zeta$ such that $\mbf L/\zeta \in \msc K$.
We call this the \emph{reflection} of $\msc K$ in $\mbf L$, and denote the congruence by $\zeta(\msc K)$.  %% or $\zeta_{\msc V}(\mbf L)$.
Thus $\mbf L/\theta \in \msc K$ if and only if $\theta \geq \zeta(\msc K)$.
%In particular, if you apply this to the finitely presented lattice $\op{FP}(\mbf S)$, you get  $\op{FP}(\mbf S)/\zeta(\msc V)$ which is the finitely presented lattice
%generated in $\msc V$ by $\mbf S$. 

On the other hand, a partially defined lattice $\mbf S$ has its own congruence lattice $\op{Con}\,\mbf S$, 
consisting of the kernels of PDL homomorphisms $h: \mbf S \to \mbf T$ where $\mbf T$ is a PDL or lattice.
PDL congruences of $\mbf S$ are equivalence relations that respect the operations (joins and meets) defined in $\mbf S$.  
To say that $\mbf S$ occurs in $\mbf L$ (ignoring $R_-$ for the time being) means that there is a PDL homomorphism
$r :  \mbf S \to \mbf L$, in which case $\ker r$ is in $\op{Con}\,\mbf S$.
We say that $\varphi \in \op{Con}\,\mbf S$ is a \emph{$\msc K$-congruence} if there exist $\mbf L \in \msc K$
and $r :  \mbf S \to \mbf L$ such that $\varphi = \ker r$.  

Observe that if $r_1 :  \mbf S \to \mbf L_1 \in \msc K$ and $r_2 :  \mbf S \to \mbf L_2 \in \msc K$,
then $r_1 \times r_2 : \mbf S \to \mbf L_1 \times \mbf L_2 \in \msc K$ with $\ker (r_1 \times r_2) = \ker r_1 \ \cap\ \ker r_2$.   
Thus there is also a least $\msc K$-congruence on $\mbf S$, which we denote by $\widehat\zeta(\msc K)$, to distinguish it
from the lattice congruence on $\mbf L$.  
It also holds that  if $r :  \mbf S \to \mbf L \in \msc K$ and $g : \mbf L \onto \mbf M$ is a surjective lattice homomorphism,
then $gr : \mbf S \to \mbf M \in \msc K$ is a PDL homomorphism.    
That does not quite prove that $\varphi$ is a $\msc K$-congruence on $\mbf S$ if and only if $\varphi \geq \widehat\zeta(\msc K)$,
but that will hold when $\msc K$ is the pseudovariety of distributive or bounded lattices, as we will see in the course of the proofs below.

%Now, let us apply these notions to the current situation.
%Assume that we are given a finite bounded lattice $\mbf L$ and a PDL homomorphism $r : \mbf S \to \mbf L$.
%Without loss of generality, the image $r(\mbf S)$ generates $\mbf L$.
To build any finite bounded lattice, you start with its distributive reflection, and then perform a series of doublings (cf.~Theorem~\ref{thm:RAD1} above).
So our first task will be to find the distributive congruence $\delta := \widehat\zeta(\msc D)$ on $\mbf S$.

\begin{lm} \label{lm:coatoms}
Let\/ $\mbf L$ be a finite bounded lattice.  The following are equivalent for a congruence $\gamma \in \op{Con}\,\mbf L$.
\begin{enumerate}[(1)]
\item  $\gamma$ is a coatom of\/ $\op{Con}\,\mbf L$.
\item  $\mbf L/\gamma \cong \mbf 2$, whence $\gamma$ is the kernel of a lattice homomorphism $g : \mbf L \onto \mbf 2$.
\item  $\gamma$ partitions $\mbf L$ into a prime filter, $\uparrow\! p$, for $p$ a join prime element,
and a prime ideal, $\downarrow\! q$, for $q$ a meet prime element. 
\end{enumerate}
Hence the distributive reflection congruence $\zeta(\msc D)$ on $\mbf L$ is the meet of the coatoms of\/ $\op{Con}\,\mbf L$.
\end{lm}
\begin{proof}
(1) implies (2) because $\mbf 2$ is the only simple finite semidistributive lattice, \emph{a fortiori}, the only simple finite bounded lattice.
Clearly (2) implies (3) implies (1).  The last statement is because $\mbf 2$ is the only subdirectly irreducible distributive lattice.
\end{proof}

Assume now that $\mbf L$ is a finite bounded lattice and $h: \mbf S \to \mbf L$ a PDL homomorphism.
Let $\gamma$ be a coatom  of $\op{Con}\,\mbf L$, so that by (2), $\gamma = \ker g$ for say $g : \mbf L \onto \mbf 2$.  
Then $r = gh : \mbf S \to \mbf 2$ is a PDL homomorphism.
Put $\widehat\gamma = \ker r$, so that:
\begin{align*}
  s \ \widehat\gamma\  t \   &\text{ iff }\  gh(s)=gh(t)  \\
                                 &\text{ iff }\  (h(s),h(t)) \in \ker g = \gamma .
\end{align*}
The $\gamma$-classes on $\mbf L$ are a prime ideal $I = g^{-1}(0)$ and a prime filter $F = g^{-1}(1)$.
Now it is possible that $h(S) \subseteq I$ or $h(S) \subseteq F$.
But if not, then $\widehat\gamma$ partitions $\mbf S$ into a prime ideal $\widehat I = h^{-1}(I)$ and prime filter $\widehat F = h^{-1}(F)$.
(These are prime with respect to the defined operations of $\mbf S$, and need not be principal.)

%% I think you're right, we don't need this paragraph
%Moreover, if $\gamma_1, \dots, \gamma_m$ are the coatoms of $\op{Con}\,\mbf L$, then
%\begin{align*}
 % (h(s),h(t)) \in \zeta(\msc D) = \bigcap \gamma_j \   &\text{ iff }\    ((h(s),h(t)) \in \gamma_j  \  \text{ for all } j  \\
  %                               &\text{ iff }\    (s,t) \in \widehat\gamma_j  \ \text{ for all } j \\
  %                              &\text{ iff }\ (s,t) \in \bigcap \widehat\gamma_j .
%\end{align*}
%Thus $\widehat\zeta(\msc D) := h^{-1}(\zeta(\msc D)) = \bigcap \widehat\gamma_j$.

Meanwhile, for any PDL homomorphism $r: \mbf S \onto \mbf 2$ whatsoever, 
the kernel $\ker r$ is a PDL congruence, partitioning $\mbf S$ into a prime
ideal and a prime filter (with respect to the operations defined  in $\mbf S$).  
Conversely, if we partition $\mbf S$ into a prime ideal and a prime filter, then there is a natural $r: \mbf S \to \mbf 2$
with that partition as its kernel.  

Assume we have found all such maps $r_1, \dots, r_\ell : \mbf S \onto \mbf 2$.  
Associated with each $r_j$ is the PDL congruence $\ker r_j$.  
Let $r : \mbf S \to \mbf 2^\ell$ be defined via $r(s)_j = r_j(s)$ for $s \in \mbf S$.  
Then $r$ is a realization of $\mbf S$ in a distributive lattice, and $\ker r = \bigcap \ker r_j$.
Moreover, for the sublattice $\mbf D$ of $\mbf 2^\ell$ generated by $r(\mbf S)$, we have $\delta = \widehat\zeta(\msc D) = \ker r$.

Given $\mbf S$, the first part of our algorithm is to find all the PDL homomorphisms $r_1, \dots, r_\ell : \mbf S \onto \mbf 2$,
and to form their intersection $\delta := \ker r = \bigcap \ker r_j$.
By the above discussion, for a pair $(u,v)$ from $R_-$, the following are equivalent:
\begin{itemize}
\item  $u \leq v$ in the finitely presented distributive lattice $\op{FP}(\mbf S)/\zeta(\msc D)$;
\item  $r(u) \leq r(v)$;
\item  $(u \join v) \ \delta\ v$, i.e., $u \leq v$ mod $\delta$.
\end{itemize}
%The conditions can be thought of as saying that ``$\mbf S \,+ $ distributivity implies $u \leq v$.''
The first part of our algorithm consists of finding $\delta$ and testing whether the above conditions apply to pairs $(u,v) \in R_-$.
If it happens that $r(u) \nleq r(v)$ for all pairs in $R_-$, then the algorithm returns YES, since finite distributive lattices are bounded.
Otherwise we continue, seeking to find the bounded PDL congruence $\beta = \widehat\zeta(\msc B) \leq \delta$ 
to determine whether $u \nleq v$ mod $\beta$ for all $(u,v) \in R_-$.
%\newline \gianluca{Add: seeking to find a PDL BOUNDED congruence $\beta \leq \delta$ such that $u \nleq v$ mod $\beta$ for all $(u,v) \in R_-$. Or something like that?}  Yes, DONE

%Recall how bounded congruences work on a finite lattice $\mbf L$.
%Corresponding to the sequence of varieties  $\msc D = \msc B_0 \subseteq \msc B_1 \subseteq \msc B_2 \subseteq \cdots$
%there are congruences $\delta \geq \zeta(\msc B_1) \geq \zeta(\msc B_2) \geq \cdots$ on $\mbf L$.
%Since $\mbf L$ is finite, this sequence terminates in a congruence $\beta = \zeta(\msc B_N)$ for some $N$,
%with the property that  $\mbf L/\theta$ is a bounded lattice if and only if $\theta \geq \beta$.  
%Moreover, if $\varphi \succ \theta \geq \beta$ in $\op{Con}\,\mbf L$, then $\mbf L/\theta$
%is obtained from $\mbf L/\varphi$ by doubling an interval.  (This is part of the proof of Theorem~\ref{thm:RAD1}.)

To design the second part of the algorithm, let us analyze how doubling an interval is reflected in $\mbf S$ when
there is a realization $r: \mbf S \to \mbf L$.
Assume now that we are given:
\begin{itemize} 
  \item  a PDL $\mbf S$ obtained from a standardized presentation;
  \item  a finite lattice $\mbf L$ and a PDL homomorphism $r : \mbf S \to \mbf L$;
  \item  an interval $I = [b,a]$ in $\mbf L$.
\end{itemize}
(For future reference call this set of conditions $(\yen)$.)
Note that $\kappa := \ker r$ is a PDL congruence on $\mbf S$.
Let $\lambda : \mbf L[I] \to \mbf L$ be the standard map.

Since $\mbf S$ is standardized, by Lemma~\ref{stabledoub} there is a PDL homomorphism $r' : \mbf S \to \mbf L[I]$
such that $r = \lambda r'$.  
In other words, there is some assignment $\alpha : S \to \{ 0,1 \}$ such that 
\[   r'(s) = \begin{cases} r(s)  &\text{ if } r(s) \notin I, \\  (r(s),\alpha(s))  &\text{ if } r(s) \in I \end{cases}    \]
is a PDL homomorphism, witnessing that $\mbf S$ occurs in $\mbf L[I]$.
Let $\kappa' = \ker r'$, noting that $\kappa'$ is a PDL congruence with $\kappa' \leq \kappa$.

Set
\begin{align*}
G &= \{ s \in S : r(s) \geq b \}  \quad \text{(a filter of $\mbf S$)}, \\
J &= \{ s \in S : r(s) \leq a \}  \quad \text{(an ideal of $\mbf S$)}, 
\end{align*}
and consider $G \cap J$.  It could be empty:  in that case, we have $r' = r$ and $\kappa' = \kappa$.

So, assume $G \cap J \ne \varnothing$.  Then, mimicing the assignment $\alpha$, put:
\begin{align*}
A_0 &= \{ s \in S : r'(s) = (r(s),0) \}, \\
A_1 &= \{ s \in S : r'(s) = (r(s),1) \}.
\end{align*}

\begin{lm} \label{lm:split}
With respect to the above setup, we have the following:
\begin{enumerate}[(1)]
\item $s \in G$ if and only if\/ $r(s) \geq b$;
\item $s \in J$ if and only if\/ $r(s) \leq a$;
\item $A_0$ is an ideal of $G \cap J$;
\item $A_0$ is meet prime in $G$;
\item $A_1$ is a filter of $G \cap J$;
\item $A_1$ is join prime in $J$;
\item $G \cap J = A_0 \ \dot\cup\ A_1$ is a union of $\kappa$-classes;
\item $s\ \kappa' \  t$ if and only if we have:
   \begin{enumerate}
   \item $s \ \kappa\ t \notin G \cap J$; or
   \item $s \ \kappa\   t$ and both are in $A_0$; or
   \item $s \ \kappa\   t$ and both are in $A_1$.
   \end{enumerate}
\end{enumerate}
\end{lm}

The proof of each item is straightforward.
In view of \ref{lm:split}(8), we say that $\kappa'$ \emph{splits} the $\kappa$-classes.
However, it can happen that $A_0 = G \cap J$ or $A_1 = G \cap J$, in which case
we again get $\kappa' = \kappa$ and not a proper refinement.

Our algorithm depends on the following lemma.

\begin{lm} \label{lm:bdedrecur}
Assume we are given a bounded lattice $\mbf L$ and a PDL homomorphism $r : \mbf S \to \mbf L$ with kernel $\kappa$. 
Assume there exist a filter $G$ of\/ $\mbf S$, and an ideal $J$ of\/ $\mbf S$, such that $G \cap J \ne \varnothing$.  
In $\mbf L$, set:
\begin{align*}
b &= \Meet r(G \cap J), \\
a &= \Join r(G \cap J),
\end{align*} 
whence $b \leq a$.  Assume
\begin{enumerate}[(1)]
\item $s \in G$ if and only if\/ $r(s) \geq b$;
\item $s \in J$ if and only if\/ $r(s) \leq a$.
\end{enumerate}
Suppose there is an assignment $\alpha : G \cap J \to \mbf 2$ such that:
\begin{enumerate}[(3)]
\item[(3)] $A_0 = \{ s \in G \cap J : \alpha(s)=0 \}$ is an ideal of $G \cap J$;
\item[(4)] $A_0$ is meet prime in $G$;
\item[(5)] $A_1 = \{ s \in G \cap J : \alpha(s)=1 \}$ is a filter of $G \cap J$;
\item[(6)] $A_1$ is join prime in $J$;
\item[(7)] $G \cap J = A_0 \ \dot\cup\  A_1$ is a union of $\kappa$-classes;
\item[(8)]  there exist $s$, $t \in G \cap J$ such that $s \ \kappa\ t$, $s \in A_0$, $t \in A_1$.
\end{enumerate}
Let $I=[b,a]$. 
Then the map $r' : \mbf S \to \mbf L[I]$ given by 
\[   r'(s) = \begin{cases} r(s)  &\text{ if } s \notin G \cap J, \\  (r(s),\alpha(s))  &\text{ if } s \in G \cap J \end{cases}    \]
is a PDL homomorphism.
Moreover,  $\kappa' = \ker r' < \ker r = \kappa$.
%\gianluca{The referee says: Also (1) of Lemma 3.4 should be $A_0 = \{s \in G \cap J : \alpha(s) = 0\}$. Similarly
%for (3). As it is $A_0 \cup A_1 = S$. Depending on how you fix this, changes on the previous and following pages may be required.} FIXED
\end{lm}

Again, the calculations are straightforward, with (8) ensuring that $\kappa' < \kappa$.

%If no $\kappa$-class is split by $\kappa'$, then forget it; the doubling is not needed.
%\newline \gianluca{Once again, explain "then forget it; the doubling is not needed".}
%Otherwise the new occurrence of $\mbf S$ in $\mbf L[I]$ is witnessed by $r'$. 
%Note that if there are two (or more) possible refinements, $\kappa_1'$ and $\kappa_2'$, we can either:
%\begin{itemize}
%\item do them separately and take the intersection; or
%\item since $\kappa_1' \cap \kappa_2'$ refines both, do them sequentially.
%\end{itemize}

Our procedure is to recursively apply Lemma~\ref{lm:bdedrecur}.
Start with $r_0 : \mbf S \to \mbf L_0 = \mbf D \leq \mbf 2^\ell$  constructed in the first part, with $\kappa = \ker r_0 = \delta$.
If the hypotheses of the lemma can be satisfied, we get $r_1 : \mbf S \to \mbf L_1 = \mbf L_0[I_0]$ with $\ker r_1 < \delta$. 
If the hypotheses can be satisfied with $r_1$ and $\mbf L_1$, apply the lemma again.
Continue until you reach $r_m$ and $\mbf L_m$ for which there is no assignment $\alpha$ satisfying the hypotheses.

Note that $\mbf L_m$ is bounded, because $\mbf L_0$ is distributive and $\mbf L_{j+1} = \mbf L_j[I_j]$ is an interval doubling.
Also $m \leq |S|-1$, because $\kappa_m < \kappa_{m-1} < \dots < \kappa_0$ is a chain of partitions of $S$, ordered by refinement. 

We claim that $\kappa_m = \beta = \widehat\zeta(\msc B)$.  
Suppose to the contrary that $\kappa_m > \beta$.  
To avoid confusion, let us proceed carefully.  %%maybe say that better
\begin{itemize}
\item  Given are $r_m : \mbf S \to \mbf L_m$ with $\ker r_m = \kappa_m$
 and say $h: \mbf S \to \mbf K$ with $\ker h = \beta < \kappa_m$.  
\item  For $\theta \in \op{Con}\,\mbf K$ let $f_\theta : \mbf K \to \mbf K/\theta$  
 and $q_\theta = f_\theta h : \mbf S \to \mbf K/\theta$.
\item  Choose $\theta$ maximal in $\op{Con}\,\mbf K$  such that $\ker q_\theta \ngeq \kappa_m$.
 Let $\mbf N = \mbf K/\theta$.
 \item Let $\xi \succ \theta$ $\op{Con}\,\mbf K$, so that $\ker q_\xi \geq \kappa_m$.
 Let $\mbf M = \mbf K/\xi$.
\item  By the proof of Theorem~\ref{thm:RAD1} in \cite{RAD77}, since $\xi \succ \theta$,
there is an interval $J$ of $\mbf M$ such that $\mbf N \cong \mbf M[J]$.
\item Note $\mbf L_m \times \mbf M[J] \cong (\mbf L_m \times \mbf M)[\mbf L_m \times J]$.
\item Now we have the setup for Lemma~\ref{lm:split}:  
$r_m \times q_\xi : \mbf S \to \mbf L_m \times \mbf M$ and
$r_m \times q_\theta : \mbf S \to (\mbf L_m \times \mbf M)[\mbf L_m \times J]$.  
\item Moreover, $\ker (r_m \times q_\xi) = \ker r_m \,\cap\, \ker q_\xi = \kappa_m$, 
while  $\ker (r_m \times q_\theta) = \ker r_m \,\cap\, \ker q_\theta < \kappa_m$.
\item  Thus a proper refinement using Lemma~\ref{lm:split} is possible, contradicting the original
assumption about $r_m$ and $\mbf L_m$.
\item We conclude that $\kappa_m = \beta$.
\end{itemize}

This gives us the algorithm for determining whether a standardized $\la \mbf S,R_- \ra$ occurs 
in a finite bounded lattice.
\begin{enumerate}
\item Find the distributive reflection congruece $\delta$ on $\mbf S$ as in the first part.
\item Refine by doubling, as in the second part, until no further refinements consistent with
$\mbf S$ are possible.  This yields the bounded reflection congruence $\beta$ on $\mbf S$.  
%\newline \gianluca{This part of the algorithm has to be explained better.}
\item  For each pair $(u,v)$ in $R_-$, we have:
\begin{itemize}
\item  if $u \leq v$ mod $\beta$, then $\la \mbf S, u \nleq v \ra$ cannot be modeled.
\item  if $u \nleq v$ mod $\beta$, then the final representation $\mbf L_m$ models it.
\end{itemize}
\end{enumerate}
If $u \nleq v$ mod $\beta$ for every pair $(u,v) \in R_-$, the algorithm returns YES; otherwise it returns NO.
%As this algorithm requires at most $|S|$ steps, we conclude:
%\newline \gianluca{Explain better why the algorithm requires at most $|S|$ steps!!!}

Thus we can finally conclude:

 \begin{thm}\label{main_theorem_2} The existential theory of infinite free lattices is decidable.   
 \end{thm}

\smallskip \noindent 
This is of course equivalent to the fact that the universal theory of free lattices is decidable, thus proving our Main Theorem \ref{main_theorem}.

\begin{proof}
A sentence $\exists \mbf x \  \psi(\mbf x)$, where $\psi$ is quantifier-free, holds in a lattice $\mbf L$ if there exists an assignment of the variables
$\varepsilon : \mbf x \to L$ such that $\psi(\varepsilon \mbf x)$ is true.
We want a procedure to decide whether such a sentence holds in some/every infinite free lattice $\mbf F_n$ with $n \geq 3$.

\smallskip \noindent The first step is to put $\psi(\mbf x)$ into disjunctive normal form.  This yields a sentence:
\[ (\dagger) \quad  \exists \mbf x \  \psi_1(\mbf x)  \,\op{OR}\, \dots \,\op{OR}\, \psi_p(\mbf x), \]
where $\mbf x = (x_1, ..., x_m)$ and each $\psi_j(\mbf x)$ is a (finite) conjunction of lattice literals, i.e., $s(\mbf x) \leq t(\mbf x)$ or $u(\mbf x) \not\leq v(\mbf x)$.
%with $s$, $t$, $u$, $v$ lattice terms in the variables $\mbf x = (x_1, ..., x_m)$.   

\smallskip \noindent For each $\psi_j(\mbf x)$, we use Skolem's Algorithm (Theorem~\ref{skolem}) to test whether the sentence is
consistent with lattice theory, i.e., whether it holds in any lattice.   If every $\psi_j(\mbf x)$ is inconsistent, then 
the answer is NO:  $\exists \mbf x \  \psi(\mbf x)$ does not occur in any lattice, much less a free one.
If some of the $\psi_j$ are consistent, then remove the inconsistent ones, and for the ones that are consistent
form a pair $\la \mbf S_j, R_{-j} \ra$ consisting of a partially defined lattice and negations.  

\smallskip \noindent These will be tested separately, to see if any one of them occurs in a free lattice.
But first we standardize each presentation, using Lemma~\ref{lm:std}.  This replaces each presentation 
$\la \mbf S_j, R_{-j} \ra$ with a disjunction of standardized presentations of the same form.  
The problem still has the form of $(\dagger)$, but with more options to be tested.

\smallskip \noindent By Theorem~\ref{thm:reduction}, a standardized configuration $\la \mbf S, R_- \ra$ occurs in a free lattice
if and only if it occurs in a finite bounded lattice.  That, in turn, in tested by the algorithm of this Section~\ref{subsec:part2}.
If the answer is YES for any of the configurations in the disjunction, then $\exists \mbf x \  \psi(\mbf x)$ occurs in a free lattice.
If the answer is NO for every option, then the overall answer is NO, and it does not occur in a free lattice.
\end{proof}

\subsection{Examples}\label{sec:examples}

%\gianluca{Add a quick introductory sentence on why we introduce these examples and say in which of the examples the algorithm return YES and on why it returns NO. For the benefit of the reader.} DONE
%\gianluca{As you said the examples have to be explained and expanded a bit but they are very nice!} DONE

The algorithm in the proof of the Main Theorem determines whether $\la \mbf S,R_- \ra$ occurs in a finite bounded lattice, or equivalently, a free lattice.  The algorithm returns YES if the configuration occurs, and NO if it does not.  This depends on what pairs $(u,v)$ are in $R_-$; after all, the positive part is modeled in a 1-element lattice.  We give five examples.

{\bf \#1.}
Consider the existential sentence 
\[ \psi_1 : \quad \exists x\exists y \exists z \quad x \leq y \join z \,\&\, y \leq x \join z \,\&\,z \leq x \join y \,\&\,x \nleq y \,\&\, \dots \,\&\, z \nleq y \]
where the last part just says that $x$, $y$, $z$ are an antichain.  
Then, after applying Skolem's algorithm, $\mbf S$ just has 4 elements:  $x$, $y$, $z$, $t=x \join y = x \join z = y \join z$.  
Here $X= \{ x,y,z \}$ but $S = \{ x,y,z,t \}$.  
As a partially defined lattice, $R_+$ has, besides $x \join y = t$ etc., the operations reflecting the order:
$x \join x = x$, $x \join t = t$, $x \meet t = x$, and so forth.  Meanwhile, $R_-$ contains the inequalities that
$x$, $y$, and $z$ are incomparable, such as $x \nleq y$.
%Under any lattice interpretation (homomorphism), $x \join y \join z = x \join t = t$.
But $x \meet y$, $x \meet z$, and $y \meet z$ are undefined,
while $x \meet y \meet z$, not corresponding to an element of $S$, is nothing at this point.  The configuration is illustrated on the left in Figure~\ref{fig:ex12}.

The previous paragraph is all setup.  
To apply the algorithm above, form the PDL $\mbf S$.
The partition $\delta_x = [ x | yzt ]$ divides $S$ into a prime ideal and a prime filter.  The PDL congruences $\delta_y$ and $\delta_z$ are symmetric.  The distributive reflection congruence $\delta = \delta_x \cap \delta_y \cap \delta_z$ is the identity equivalence relation $[x|y|z|t]$ on $S$.  Moreover $x$, $y$, $z$ are an antichain in $\mbf S/\delta$.  It follows that the sentence $\psi_1$ can be modeled in the distributive lattice $\mbf S/\delta \cong \mbf 2^3$.  The second part of the algorithm, refining the distributive reflection $\delta$ to the bounded reflection $\beta$, is skipped since $\delta$ is already the identity congruence on $\mbf S$, whence $\delta=\beta$.  
The algorithm returns YES.

{\bf \#2.}  Consider the sentence 
\[ \psi_2 : \quad \exists x\exists y \exists z \quad  z \leq x \join y \ \&\ y \leq x \join z \ \&\  y \nleq x \join (y \meet z) \]
Again following Skolem, we form the PDL $\mbf S = \{ x,y,z,t,m,s \}$ with the defined relations
\[   x \join y = x \join z = t  
\qquad\quad y \meet z = m \qquad\quad x \join m = s \]
The configuration is illustrated on the right in Figure~\ref{fig:ex12}.  Only the order and the above operations are defined in $\mbf S$:  for example, $y \join z$ is undefined.

First find all the prime-ideal/prime-filter combinations.
Intersecting those gives the distributive reflection congruence
$\delta = [ x|y|z|m|st ]$ representing $\mbf S$ in a distributive lattice, indicated by red curves in the figure. 
Modulo $\delta$, we have 
\[  y \leq t \ \delta\  s = x \join (y \meet z) \]
so $R_-$ is not satisfied.  That being the case, we attempt to use the second part of the algorithm to refine $\delta$, that is,
to split the class $[ st ]$ into $A_0 \ \dot\cup\  A_1$.
%\newline \gianluca{Explain what you mean by "We would like" with respect to the instructions given by our algorithm.}
Suppose that were possible.
Because $x \join y = t$ in $J =\, \downarrow\! t$, it would force $t \notin A_1$, and hence $s \leq t \in A_0$, a contradiction.
Thus no proper bounded refinement is possible, whence $\delta = \beta$, and the algorithm returns NO.

This example is a roundabout way of saying that the join semidistributive law holds in a bounded lattice, so that $t = x \join y = x \join z$ implies $t =  x \join (y \meet z)$.  
%\newline \gianluca{$y \nleq x \join (y \meet z)$ is in $R_-$ or NOT?! This examples should be explained better!} FIXED

\begin{figure}
\begin{center}
\begin{tikzpicture}
\tikzstyle{every node}=[scale=1.2,draw,circle,fill=white,minimum size=5pt,inner sep=0pt,label distance=1.2mm] 
    \node (0) at (-1,0) [label=below:$x$] {};
    \node (1) at (0,0) [label=below:$y$] {};
    \node (2) at (1,0) [label=below:$z$] {};
    \node (3) at (0,1) [label=right:$t$] {};

   \draw (0)--(3);
   \draw (1)--(3);
   \draw (2)--(3);
    \draw [color=red, thick] (0) circle (.2);  %%circle
    \draw [color=red, thick] (1) circle (.2);  %%circle
    \draw [color=red, thick] (2) circle (.2);  %%circle
    \draw [color=red, thick] (3) circle (.2);  %%circle
\end{tikzpicture}
\qquad\quad
\begin{tikzpicture}
\tikzstyle{every node}=[scale=1.15,draw,circle,fill=white,minimum size=5pt,inner sep=0pt,label distance=1.2mm] 
    \node (0) at (-1,1) [label=left:$s$] {};
    \node (1) at (0,1) [label=right:$y$] {};
    \node (2) at (1,1) [label=right:$z$] {};
    \node (3) at (0,2) [label=160:$t$] {};
    \node (4) at (-2,0) [label=left:$x$] {};
    \node (5) at (0,0) [label=right:$m$] {};

   \draw (5)--(0)--(3);
   \draw (5)--(1)--(3);
   \draw (5)--(2)--(3);
   \draw (4)--(0);
    \draw [color=red, thick] (4) circle (.2);  %%circle
    \draw [color=red, thick] (1) circle (.2);  %%circle
    \draw [color=red, thick] (2) circle (.2);  %%circle
    \draw [color=red, thick] (5) circle (.2);  %%circle
    \draw [color=red, rounded corners, thick] (-1,1-.25) -- (0+.25,2) -- (0,2+.25) -- (-1-.25,1) -- cycle;
\end{tikzpicture}

\caption
{Example \#1 (left).  The defined relations are $x \join y = x \join z = y \join z = t$. 
Example \#2 (right).  The defined relations are $x \join y = x \join z = t$, $y \meet z = m$, $x \join m = s$. 
Red curves indicate the distributive congruence $\delta$.}
\label{fig:ex12}
\end{center}
\end{figure}

{\bf \#3.}  
Consider  
\[ \psi_3 : \quad \exists x\exists y \exists z \quad  
x \meet z \leq y \leq x \leq y \join z \ \&\  x \nleq y  \]
These are the relations that define a pentagon, with $x \nleq y$ saying that in fact $y < x$, so that the critical quotient does not collapse.
Adding the variable names $m = x \meet z$ and $t = y \join z$,
form the PDL $\mbf S = \{ x,y,z,m,t \}$ with those defining 
relations, as illustrated on the left in Figure~\ref{fig:ex34}.
It is straightforward to calculate that the distributive congruence is $\delta = [ xy|z|m|t ]$.  However, $\mbf S/\delta$ does not satisfy $R_-$ because $x \ \delta\  y$, so we try to refine $\delta$ using the second part of the algorithm.
In fact, 
the congruence class $[ xy ]$ can be split by the doubling 
with  $A_0 = \{ y \}$ and $A_1 = \{ x \}$, the dashed line in the figure.  That gives the bounded reflection $\beta = [ x|y|z|m|t ]$ with $\mbf S/\beta$ satisfying $\psi_3$, as $y < x$ mod $\beta$.
The algorithm returns YES, and indeed it has constructed the pentagon as a model of $\psi_3$.

{\bf \#4.}  Next consider
\[ \psi_4 : \quad \exists x\exists y \exists z \quad 
x \join y = x \join z  \ \&\  y \meet x = y \meet z 
\ \&\  y \nleq z  \ \&\  z \nleq x  \]
Form the PDL $\mbf S = \{ x,y,z,m,t \}$ with $x \join y = x \join z = t$ and $y \meet x = y \meet z = m$.  
The configuration is illustrated on the right in Figure~\ref{fig:ex34}, but note that $y \join z$ and $x \meet z$ are undefined. 
The distributive reflection is $\delta = [my|z|xt ]$.
As indicated in the figure, $\mbf S/\delta$ is a 3-element chain that does not satisfy $R_-$.

Hence we proceed to the second part of the algorithm, looking for a bounded refinement of $\delta$.
But the lower class $[my]$ cannot split because if $y \in A_1$, then $y \meet z = m$ would imply $m \in A_1$ (as $A_0$ is meet prime in $G$).  Dually, neither can the upper class $[xt]$ be split.  
We conclude that $\delta = \beta$, not satisfying $\psi_4$, and the algorithm returns NO.
This example is a variation of $\#2$ with both semidistributive laws.

\begin{figure}
\begin{center}
\tikzstyle{every node}=[scale=1.,draw,circle,fill=white,minimum size=5pt,inner sep=0pt,label distance=1.2mm] 
\begin{tikzpicture}
   
    \node (0) at (0,0) [label=right :$m$] {};
    \node (1) at (1.,1.5) [label=right:$z$] {};
    \node (2) at (0,3) [label=right:$t$] {};
    \node (4) at (-1,1) [label=left:$y$] {};
    \node (5) at (-1,2) [label=left:$x$] {};

   \draw (0)--(1)--(2)--(5)--(4)--(0);
    \draw [color=red, thick] (0) circle (.2);  %%circle
    \draw [color=red, thick] (1) circle (.2);  %%circle
    \draw [color=red, thick] (2) circle (.2);  %%circle
    \draw [color=red, rounded corners, thick] (-1-.2,1-.2) -- (-1+.2,1-.2) -- (-1+.2,2+.2) -- (-1-.2,2+.2) -- cycle;
    \draw[dashed, thick] (-1.5,1.5)--(-.5,1.5);
\end{tikzpicture}
\qquad\quad
\begin{tikzpicture}
\tikzstyle{every node}=[scale=1.15,draw,circle,fill=white,minimum size=5pt,inner sep=0pt,label distance=1.2mm] 
    \node (0) at (-1,1) [label=left:$x$] {};
    \node (1) at (0,1) [label=right:$z$] {};
    \node (2) at (1,1) [label=right:$y$] {};
    \node (3) at (0,2) [label=160:$t$] {};
%    \node (4) at (-2,0) [label=left:$x$] {};
    \node (5) at (0,0) [label=-20:$m$] {};

   \draw (5)--(0)--(3);
   \draw (5)--(1)--(3);
   \draw (5)--(2)--(3);
    \draw [color=red, thick] (1) circle (.2);  %%circle
    \draw [color=red, rounded corners, thick] (-1,1-.25) -- (0+.25,2) -- (0,2+.25) -- (-1-.25,1) -- cycle;
    \draw [color=red, rounded corners, thick] (0,-.25) -- (1+.25,1) -- (1,1+.25) -- (-.25,0) -- cycle;
\end{tikzpicture}

\caption
{Example \#3 (left).  The defined relations are $m = x  \meet z \leq y \leq x \leq y \join z = t $. 
Example \#4 (right).  The defined relations are $x \join y = x \join z = t$, $y \meet x = y \meet z = m$. 
Red curves indicate the distributive congruence $\delta$.}
\label{fig:ex34}
\end{center}
\end{figure}

{\bf \#5.}  This example shows that more than one doubling may be required in the second part of our algorithm.
Consider
\begin{align*}
\psi_5 : \quad &\exists s\exists p \exists q \exists r \exists u \quad 
s \leq p \leq r  \ \&\   u \leq q \ \&\  
p \leq s \join q \ \&\   p \nleq s \join u  \ \&\ 
\\
&q \leq u \join r \ \&\  q \nleq u \join p  \ \&\ 
p = r \meet (p \join q) \ \& \ u = q \meet (u \join p) \ \&\    
\\
&s = p \meet (s \join u) 
\end{align*}
Form the PDL $\mbf S = \{ s,p,r,q,u,v,w,x,t  \}$ with $v = s \join u$, $w = u \join p$, $x= p \join q =s \join q$, and $t = q \join r = u \join r$.
The configuration is illustrated in Figure~\ref{fig:ex5}.
Some additional relations can be derived, including $r \meet v = r \meet x \meet v = p \meet v = s$.
Note that $R_- = \{ (p,v), (p,w) \}$.
One can check that the distributive reflection is $\delta = [sp|r|qu|vwx|t]$.  Since $p \ \delta\  s \leq v$ and $q \ \delta\ u \leq w$,
$R_-$ is not satisfied in $\mbf S/\delta$, so we continue to the second part of the algorithm.

One cannot separate $[sp]$ without first separating $x$ and $v$, because $r \meet x = p$ and $r \meet v = s$.
But one can split $[qu]$ as long as we also separate $x$ and $w$, since $q \meet x = q$ and $q \meet w = u$.
Splitting along the purple dashed lines gives $\beta_1 = [sp|r|q|u|vw|x|t]$.
Of course $\beta_1 < \delta$, but it still contains $[sp]$.
Again we try to split $[sp]$.  The previous objection is removed, because $(w,x) \notin \beta_1$, and we can
split along the blue dashed lines, since we must also separate $v$ and $w$.
That yields $\beta_2 = [s|p|r|q|u|v|w|x|t]$.

Now $\mbf S/\beta_2$ models $\la \mbf S,R_- \ra$, and the algorithm returns YES.

\begin{figure}
\begin{center}
\tikzstyle{every node}=[draw,circle,fill=white,minimum size=5pt,inner sep=0pt,label distance=1mm] 
\begin{tikzpicture} %[scale=.82]
 %   \node (0) at (0,0) [] {};
    \node (1) at (-1,1) [label=210:$s\,$] {};
    \node (2) at (1,1) [label=-30:$\,u$] {};
    \node (3) at (-2,2) [label=210:$p\,$] {};
    \node (4) at (0,2) [label=30:$\,v$]{};
    \node (5) at (2,2)[label=-30:$\,q$] {};
    \node (6) at (-1,3)[label=left:$w\ $] {};
    \node (7) at (0,4) [label=30:$\,x$]{};
    \node (8) at (-3,3)  [label=210:$r\,$] {};
    \node (9) at (-1,5)[label=30:$\,t$] {};
   \draw (1)--(3)--(6)--(7)--(5)--(2);
    \draw (2)--(4)--(6);
    \draw (1)--(4);
    \draw (3)--(8)--(9)--(7);
 \draw [color=red, thick] (8) circle (.2);  %%circle
 \draw [color=red, thick] (9) circle (.2);  %%circle
 
     \draw[dashed, very thick,color=purple] (1.5-.33,1.5+.33)--(1.5+.33,1.5-.33);
     \draw[dashed, very thick,color=purple] (-.5-.33,3.5+.33)--(-.5+.33,3.5-.33);
     \draw[dashed, very thick,color=blue] (-1.5-.33,1.5-.33)--(-1.5+.33,1.5+.33);
     \draw[dashed, very thick,color=blue] (-.5-.33,2.5-.33)--(-.5+.33,2.5+.33);
 
    \draw [color=red, rounded corners, thick] (1,1-.25) -- (2+.25,2) -- (2,2+.25) -- (1-.25,1) -- cycle;
    \draw [color=red, rounded corners, thick] (-1+.25,1) -- (-2,2+.25) -- (-2-.25,2) -- (-1,1-.25) -- cycle;
    \draw [color=red, rounded corners, thick] (0+.25,2) -- (-1+.25,3) --(0+.25,4)--(0,4+.25)-- (-1-.25,3) -- (0,2-.25) -- cycle;
 \end{tikzpicture}
   
\caption
{Example \#5.  The defined relations are given in the text.
Red curves indicate the distributive congruence $\delta$, which is then refined by first splitting 
across the purple dashed lines ($\beta_1$), then across the blue dashed lines, yielding $\beta_2$.}
\label{fig:ex5}
\end{center}
\end{figure}

As a nice extension of Example~\#1, we have this application of the algorithm.

\begin{cor}\label{cor:onlijoin}
Let $\psi(\mbf x)$ be a consistent sentence
\[   \exists x_1 \cdots \exists x_m \quad s_1 \leq t_1 \ \&\  \dots \ \&\  s_k \leq t_k \ \&\  u_1 \nleq v_1 \ \&\  \dots u_\ell \nleq v_\ell  \]
where $s_i$, $t_i$, $u_i$, $v_i$ are lattice terms in $\mbf x = (x_1, ..., x_m)$.
Suppose that each of the terms in $\psi(\mbf x)$ involves only variables and the join operation. 
Then $\psi(\mbf x)$ occurs in a finite distributive lattice, and hence also in a free lattice.     
\end{cor}
%% ``occurs'' means ``is modeled in some'' - do we have to say that?

Of course the same conclusion holds for sentences involving only meets.

\begin{proof}
Consistency is always required for Skolem's algorithm to produce a PDL $\mbf S$.
Standardization with these hypotheses involves only replacing $\Join p_i \nleq q$ with a disjunction 
$p_1 \nleq q \,\op{OR}\, \cdots \,\op{OR}\,  p_\ell \nleq q$; see Definition~\ref{def_standard}.
Consistency (within lattice theory) just means that $p_j \nleq q$ in $\mbf S$ for at least one $j$. 
So we are applying our algorithm above to a PDL that has only joins defined, i.e., no nontrivial meets.
In that case, for each $s \in S$, $\downarrow\! s$ is a prime ideal and $S \,\setminus \downarrow\! s$ is a prime filter.
Thus the intersection $\delta$ is the identity relation on $\mbf S$.
Moreover, $u_j \nleq v_j$ for each pair in $R_-$ by Skolem's algorithm and consistency.
Hence $\mbf S/\delta \cong \mbf S$ models $\la \mbf S,R_- \ra$.
The algorithm returns YES for $\psi(\mbf x)$.
\end{proof}

\section{Non-elementary projective lattices}

    In this section we prove Theorem~\ref{intermediate_theorem}. Let $n \in \mathbb{N}$, a lattice $\mathbf{L}$ is said to have \textit{breadth at most $n$} if whenever $a \in L$ and $S$ is a finite subset of $L$ such that $a = \bigvee S$, there is a subset $T$ of $S$, with $a = \bigvee T$ and $|T| \leq n$. The breadth of a lattice is the least $n \in \mathbb{N}$ such that it has breadth at most $n$ (if such an $n$ exists and infinite otherwise). The reader can verify that this concept is self dual. It is easy to see that having breadth $\leq n$ is expressible by a universal sentence, and so having breadth $> n$ is expressible by an existential one.

    \begin{fact}[{\cite[Theorem~1.30]{FJN}}]\label{the_fact_for_breadth} 
Let $\mathbf{L}$ be a lattice satisfying (W). Suppose elements $a_1$, $a_2$, $a_3$, and $v \in L$ satisfy:
\begin{enumerate}[(1)]
    \item $a_i \not\leq a_j \vee a_k \vee v$ whenever $\{i,j,k\} = \{1,2,3\}$;
    \item $v \not\leq a_i$ for $i = 1, 2, 3$;
    \item $v$ is meet irreducible.
\end{enumerate}
Then $\mathbf{L}$ contains a sublattice isomorphic to $\mathbf{FL}_3$.
\end{fact}

    \begin{proof}[Proof of Theorem~\ref{intermediate_theorem}] Suppose $\mathbf{L}$ satisfies  (W). Suppose now that $\mathbf{L}$ has breadth $\geq 5$, then we can find $a, a_1, ..., a_5 \in \mathbf{L}$ such that $a = a_1 \vee \cdots \vee a_5$ holds and that this join is irredundant. Then, since every element of a lattice which satisfies (W) must be either join or meet irreducible, $v = a_4 \vee a_5$ is meet irreducible. Then by Fact~\ref{the_fact_for_breadth} $\mathbf{L}$ has a sublattice isomorphic to $\mathbf{FL}_3$.
    \end{proof}

    \begin{proof}[Proof of Corollary~\ref{intermediate_theorem_cor}] If $\mathbf{L}$ is infinite and projective, then it is a sublattice of a free lattice and it satisfies $(W)$.  Hence by \ref{intermediate_theorem} we are done, as $L$ is bi-embeddable with an infinite free lattice.
    \end{proof}

%If G is projective and infinite, it need not contain $F_3$.  The example
%is $FL(2+2)$,
%the lattice freely generated by two $2$-element chains.  That may be
%almost the only
%example, though.  Your idea is definitely worth pursuing, and I will let
%you know
%if anything comes of it.  Projective lattices are unions of finitely
%many intervals,
%each of which can be embedded in a free lattice.  It would be very hard
%for all those
%intervals not to contain $FL(1+1+1)$ if it were infinite and not $FL(2+2)$.

%Can we give a clean characterization of the projective lattices that embed $F_3$? Let's call them special for the time being. Now, the %most important question is: can we separate the E-theories of $F_3$ and a projective lattice that is not special (e.g. 
%$F(2+2)$)?

\section{The positive theory of free lattices}

In the paper \cite{NaPa14} we showed that the positive theory of the free lattice $\mbf F_3$ properly contains the positive theory of $\mbf F_4$.
This note intends to do the same for $\mbf F_n$ vs.~ $\mbf F_{n+1}$ for all finite $n \geq 3$.
%%Some small sections are copied from the paper, perhaps with minor modification, indicated in orange.
%
%\textcolor{orange}{Copied}
Let $\op{PTh}(\mbf L)$ denote the positive first-order theory of $\mbf L$.
Since $\mbf F_n$ is a homomorphic image of $\mbf F_{n+1}$, we have $\op{PTh}(\mbf F_n)  \supseteq \op{PTh}(\mbf F_{n+1})$. 
To distinguish them, we seek a positive sentence $\pi^n(\mbf{x})$ that holds in $\mbf F_n$ but not in $\mbf F_{n+1}$, thus:
$$\op{PTh}(\mbf F_1)  \supsetneq \op{PTh}(\mbf F_{2}) \supsetneq \cdots \supsetneq \op{PTh}(\mbf F_{n}) \supsetneq \op{PTh}(\mbf F_{n+1}) \supsetneq \cdots$$

In this section and this section only we use the $\{ +, \cdot \}$ notation rather than the $\{ \vee, \wedge \}$ notation as this notation is preferable for some longer expressions of lattice terms appearing below.

%\section{Preliminaries}
\subsection{Preliminaries}
The cases $n=1$ and $n=2$ are easy:
\begin{align*}
    \pi^1 = \qquad &\forall x \forall y \  x \approx y \\
    \pi^2 = \qquad &\forall x \forall y \ [ x \approx y  
    \ \text{OR}\ \forall z \ (z \leq x+y)  
    \ \text{OR}\ \forall z \ (z \geq x\cdot y) ]
\end{align*}
are sentences with the property that $\pi^n$ holds in $\mbf F_n$ but not $\mbf F_{n+1}$ for $n=1,2$.  So henceforth we assume $n \geq 3$.

As part of his solution of the word problem \cite{PMW1941}, Whitman showed that free lattices satisfy the following condition:
\[ \mathrm{(W)} \quad s \meet t \leq u \join v \text{ implies } s \leq u \join v \text{ or } t \leq u \join v \text{ or }s \meet t \leq u \text{ or } s \meet t \leq v .    \]

A free lattice $\mbf F_n$ has a unique set of generators, say $X = \{ x_1, \dots, x_n \}$.
These generators are both join and meet prime, which means:
\begin{itemize}
\item if $x_j \leq \sum U$ for some finite $U \subseteq F_n$, then $x_j \leq u$ for some $u \in U$;
\item if $x_j \geq \prod V$ for some finite $V \subseteq F_n$, then $x_j \geq v$ for some $v \in V$.
\end{itemize}
In particular, the set $X$ of free lattice generators is \emph{independent}:  $x_i \nleq \sum_{j \neq i} x_j$ and
$x_i \not\geq \prod_{j \neq i} x_j$.
In a lattice satisfying Whitman's condition $\mathrm{(W)}$, any independent set $Y$ generates
a sublattice isomorphic to $\mbf F(Y)$, see Section~1.2 of \cite{FJN} for a full exposition of these ideas. In what follows recall the notions introduced in \ref{bounded_sec}.

%By a \emph{bounded homomorphism} we mean a homomorphism 
%$f: \mbf K \to \mbf L$, where $\mbf K$ and $\mbf L$ are finitely generated lattices, such that every congruence class of $\ker f$ has a greatest and least %element.  A \emph{bounded lattice} is a finitely generated lattice $\mbf L$ with the property that every homomorphism $f : \mbf K \to \mbf L$, with $\mbf K$ finitely generated, is bounded.  Bounded lattices were characterized by McKenzie~\cite{McK1972} and Kostinsky~\cite{Kostinsky}.  They play an important role in the theory of free lattices; see Section~II.1 of~\cite{FJN}, especially Theorem~2.13.

%The term \emph{bounded lattice} is sometimes used to refer to lattices with $0$ and $1$ in the signature.  In this note we are dealing with finitely generated lattices.  If $\mbf L$ is generated by a finite set $X$, then $\mbf L$ will have a least element $\prod X$, which we will denote by $0$, but $0$ is not a term in the language.  Likewise, $1 = \sum X$ is the greatest element of $\mbf L$, but not a term in the language.

Every element in a finite distributive lattice $\mbf D$ can be written uniquely
as a join of join irreducible elements, and likewise as a meet of meet irreducible elements.
%%\gianluca{The referee says: The “or” here is not clear. It makes it sound like youonly get one or the other.}  FIXED 
The former is often known as the \emph{join normal form} and
the latter as the \emph{meet normal form}.  We are particular interested
in the case when $\mbf D$ is the free distributive lattice $\mbf{FD}_n$.

Accordingly, let $Z = \{ z_1, \dots, z_n \} = \mbf{z}$ and let 
$f_n : \mbf F(Z) \to \mbf{FD}(Z)$
be the standard homomorphism with $f_n(z_i) = z_i$ for $i \in [1,n]$.  
Now $\mbf{FD}_n$ is a (finite) bounded lattice. Let $\mbf F(Z) = \mbf{F}_n$ and $\mbf{FD}_n = \mbf{FD}(Z)$.
Thus $f_n$ is a bounded homomorphism, i.e., the $\ker f_n$-class
that is the preimage of an element $d \in \mbf{FD}_n$ 
is a bounded interval $[ \nu_d(\mbf z), \mu_d(\mbf z) ]$ in $\mbf F_n$.  That is, for $w = w(\mbf z) \in \mbf F_n$ we have:
\[ f_n(w)=d \quad\text{iff}\quad  \nu_d(\mbf z) \leq w \leq \mu_d(\mbf z) .\]
Moreover, $\nu_d(\mbf z)$ is a join of meets of variables, and dually $\mu_d(\mbf z)$ is a meet of joins of variables.  
Not surprisingly, $\nu_d(\mbf z)$ and $\mu_d(\mbf z)$ are the lattice terms in $\mbf F(Z)$ corresponding to the join normal form and meet normal form of $d$ in $\mbf {FD}(Z)$, respectively. Notice that $\mbf {FD}(Z) \models \nu_d(\mbf z) = \mu_d(\mbf z)$ but in general $\mbf F(Z) \models \nu_d(\mbf z) \neq \mu_d(\mbf z)$, unless $d$ is either $0$, $1$, an atom, a coatom, or a generator; thus, in these cases $d$ has a unique preimage under $f_n$. For example, let $d$ be the upper cover of $z_1$ in $\mbf{FD}_3$. 
Then $\nu_d(\mbf z) = z_1 + z_2z_3$ and $\mu_d(\mbf z) = (z_1 + z_2)(z_1+z_3)$, which are the join and meet normal forms (respectively) of the upper cover of $z_1$ in $\mbf{FD}_3$, and in fact in this case we have that $\mbf {FD}(Z) \models \nu_d(\mbf z) = \mu_d(\mbf z)$ but $\mbf F(Z) \models \nu_d(\mbf z) \neq \mu_d(\mbf z)$.
That the bounds of the $f_n^{-1}$-classes are join and meet normal forms reflects the fact that join irreducible elements are join prime in distributive lattices, and, dually, meet irreducibles are meet prime.  
When $\mbf L$ is a non-distributive lattice and $g: \mbf F \to \mbf L$ is a bounded homomorphism, the algorithm for finding the least element of $\ker g$-classes (Theorem~2.4 of \cite{FJN}) includes terms for nontrivial join covers of join irreducible elements, $a \leq \sum B$ in $\mbf L$ with $a \nleq b$ for all $b \in B$.  There are no such nontrivial inclusions in $\mbf {FD}(Z)$, and consequently no such terms in $\nu_d(\mbf z)$, the least element of $f_n^{-1}(d)$ when $d$ is join irreducible.  Moreover, general considerations show that the least elements of congruence classes satisfy $\nu_{\sum a_i}(\mbf z) = \sum \nu_{a_i}(\mbf z)$, so when $\mbf L$ is finite it suffices to compute $\nu_a(\mbf z)$ for $a$ join irreducible.  The dual statements hold for the largest elements $\mu_a(\mbf z)$ of congruence classes.

%Note that the following elements of $\mbf{FD}_n$ have a unique preimage under $f_n$:  0, 1, atoms, coatoms, generators.

%Throughout we use boldface to denote $k$-tuples, 
%so that $\mbf x = (x_1, \dots, x_k)$, etc.
%The notation $\downarrow\! x$ denotes $\{ u \in L : u \leq x \}$, and dually $\uparrow\! x$ denotes $\{ v \in L : v \geq x \}$.

\subsection{Some auxiliary formulas}
%\textcolor{orange}{Copied}
The negation of independence is a positive property.
Likewise, being the greatest or least element of a lattice are positive properties.
Thus we have the following positive first-order formulas in the language of lattice theory:
\[  \op{NI}(x_1, \dots, x_n): \quad (\op{OR}_{1 \leq i \leq n}  x_i \leq \sum_{j \ne i} x_j) \quad \op{OR} \quad 
             ( \op{OR}_{1 \leq i \leq n}  x_i \geq \prod_{j \ne i} x_j);   \]
\[  t(u): \quad \forall w \ w \leq u; \]
\[  b(u): \quad \forall w \ w \geq u. \]
We need one more type of formula.
Let $\mbf x = (x_1, \dots, x_{n+1})$ and $\mbf z = (z_1, \dots, z_n)$.  
Let $E = (I_1, \dots, I_{n+1})$ be a vector of bounded intervals, say $I_j = [b_j,a_j]$, where each $b_j(\mbf z)$ and $a_j(\mbf z)$ is given by a lattice term in the variables $\mbf z$.
Because the intervals are bounded, the condition
$x \in I = [b,a]$ can be written out as $b \leq x \ \&\  x \leq a$.  The condition that $x_j \in I_j$ for $1 \leq j \leq n+1$ can thus be written as a positive first-order predicate:
\[ \op{Loc}(\mbf x,E,\mbf z): \quad \&_{1 \leq j \leq n+1}\  
                 [ b_j(\mbf z) \leq x_j \leq a_j(\mbf z) ].\]
These four kinds of predicate are combined in the sentence $\pi^n$, given just before Lemma~\ref{lm:hold4}.  
The formula $\op{NI}(\mbf x)$ is a positive way of expressing that these elements are \emph{not independent}; $t(u)$ means that $u$ is the top element; $b(u)$ means that $u$ is the bottom element.
For a vector $E$ of intervals, $\op{Loc}(\mbf x,E,\mbf z)$ says that each $x_j$ is \emph{located} in the interval $I_j$.  
The task between here and Lemma~\ref{lm:hold4} is to explain how to choose properly the set $\msc E$ of vectors of intervals to be used in $\pi^n$.

%\textcolor{orange}{Copied}
In the presence of Whitman's condition $\mathrm{(W)}$,  $\op{NI}(x_1, \dots, x_k)$ says exactly that those elements
do not generate a copy of $\mbf F_k$ by \cite[Corollary~1.12]{FJN}.
On the other hand,
Whitman showed that $\mbf F_3$ contains $\mbf F_\omega$ \cite{PMW1941}, see
\cite[Theorems~1.28 and~9.10]{FJN}.
Indeed, by Tschantz's Theorem (\cite[Theorem~9.10]{FJN}), every infinite interval in $\mbf F_n$ contains a copy of $\mbf F_\omega$.
Thus $\mbf F_n$ contains many copies of $\mbf F_{n+1}$, and the predicate $\op{Loc}(\mbf x,E,\mbf z)$
%, which will be part of the formulas $\pi^n$ defined below, 
will restrict the location of copies of $\mbf F_{n+1}$.
By way of comparison, $\mbf F_n$ contains only finitely many finite intervals except for 2- and 3-element chains associated with completely join irreducible elements; see Theorem~7.5 
of \cite{FJN}.  The exceptions are in the connected components of 0 and 1; see Case~3 in the description of 
$\ker h_n$-classes below.  The location of infinite intervals motivated our use of $\mbf A_n$ rather than $\mbf{FD}(n)$.
%\gianluca{The referee says: Also you might want to give the heuristic for $NI$ and $CJ(x, E, z)$.}

\subsection{The lattice $\mbf A_n$}
We will take $\mbf F_n$, $\mbf{FD}_n$, and $\mbf A_n$ to be generated by the set $Z = \{ z_1, \dots, z_n \}$.  
To emphasize this we sometimes write $\mbf F_n = \mbf F_n(Z)$.

The lattice $\mbf A_n$ is obtained by taking the free distributive lattice $\mbf{FD}_n$ and doubling the elements that
are the join of two atoms, or the meet of two coatoms.  
The lattice $\mbf A_3$ is drawn in Figure~\ref{fig:doubled}.
Similarly, $\mbf A_4$ would have $166+12=178$ elements\footnote{Finitely generated free distributive lattices are finite, but quite large.
The cardinality $|\mbf{FD}_n|$ is known only for $n \leq 9$, with the case $n=9$ being
found only in 2023.  See the entry for \emph{Dedekind numbers} in OEIS (https://oeis.org/A000372);
the numbers given there are $|\mbf{FD}_n|+2$ since they are for the free algebra in the variety
of distributive lattices with $0$ and $1$ as constants.  The simplest asymptotic formula is due 
to D. J. Kleitman (Proc. Amer. Math. Soc. {\bf 21} (1969),  677--682):
\[ \log_2 |\mbf{FD}_n| \sim \binom n {\lfloor n/2 \rfloor}.  \]
}.

The free distributive lattice $\mbf{FD}_n$ is a finite bounded lattice.
Also, Alan Day's doubling construction preserves the property of being bounded \cite{RAD77, RAD79}, \cite [Sec.~II.3]{FJN}.
Thus the natural homomorphism $h_n : \mbf F_n \to \mbf A_n$ is a bounded homomorphism. The algorithm for computing lower and upper bounds 
for congruence classes of the kernel of a bounded homomorphism, which goes back to J\'onsson \cite{JN77} and 
McKenzie \cite{McK1972}, is given in Theorem 2.3 of \cite{FJN}.
Applying this to the homomorphism $h_n$ decomposes $\mbf F_n$ into a disjoint union of
the finitely many congruence classes of $\ker h_n$.  
The $\ker h_n$ classes will be intervals $[\beta_a(\mbf z),\alpha_a(\mbf z)]$ in $\mbf F_n$, for each $a \in A_n$,
with the property that for $w = w(\mbf z) \in \mbf F_n$:
\[ h_n(w)=a \quad\text{iff}\quad  \beta_a(\mbf z) \leq w \leq \alpha_a(\mbf z) .\]
Indeed $h_n(w) \geq a$ iff $w \geq \beta_a(\mbf z)$, and dually.
We will describe those intervals momentarily.

\begin{figure}
\begin{center}
\begin{tikzpicture}[scale=1]
\tikzstyle{every node}=[draw,circle,fill=white,minimum size=5pt,inner sep=0pt,label distance=1mm] 
   \node (0) at (0,0)[]{};
   \node (1) at (-1,1)[]{};
   \node (2) at (0,1)[]{};
   \node (3) at (1,1)[]{};
   \node (4) at (-1,2)[]{};
   \node (5) at (0,2)[]{};
   \node (6) at (1,2)[]{};
   \node (7) at (-1,3)[]{};
   \node (8) at (0,3)[]{};
   \node (9) at (1,3)[]{};
   \node (10) at (0,4)[]{};
   \node (11) at (-1,1+4)[]{};
   \node (12) at (0,1+4)[]{};
   \node (13) at (1,1+4)[]{};
   \node (14) at (-1,2+4)[]{};
   \node (15) at (0,2+4)[]{};
   \node (16) at (1,2+4)[]{};
   \node (17) at (-1,3+4)[]{};
   \node (18) at (0,3+4)[]{};
   \node (19) at (1,3+4)[]{};
   \node (20) at (0,4+4)[]{};
   \node (21) at (-2,4)[label=left:$z_1$]{};
   \node (22) at (1,4)[label=right:$z_3$]{};
   \node (23) at (2,4)[label=right:$z_2$]{};
        
   \draw (10)--(9)--(6)--(3)--(0)--(1)--(4)--(7)--(10)--(8)--(5);
   \draw (0)--(2)--(4);
   \draw (2)--(6);
   \draw (1)--(5)--(3);
   \draw (10)--(11)--(14)--(17)--(20)--(19)--(16)--(13)--(10)--(12)--(15);
   \draw (20)--(18)--(16);
   \draw (18)--(14);
   \draw (19)--(15)--(17);
   \draw (7)--(21)--(11);
   \draw (8)--(22)--(12);
   \draw (9)--(23)--(13);
\end{tikzpicture}
\end{center}
\caption
%\textcolor{orange}{Copied}
{Lattice $\mbf A_3$ obtained by doubling six elements in $\mbf{FD}_3$.} \label{fig:doubled} 
\end{figure}

The $\ker h_n$-classes come in three forms.
It is useful to describe these in terms of the maps
\begin{align*}
    h_n &: \mbf F_n \to \mbf A_n  \\
    g_n &: \mbf A_n \to \mbf{FD}_n \\
    f_n = g_nh_n &: \mbf F_n \to \mbf{FD}_n 
\end{align*}
where $g_n$ is the canonical homomorphism collapsing every
doubled point back to a singleton.
All these maps have $z_i \mapsto z_i$ for $i \in [1,n]$.
Recall that $\nu_d(\mbf z)$ and $\mu_d(\mbf z)$ denote the join normal form and meet normal form, respectively, of an element $d \in \mbf{FD}(Z)$, recall also that $\nu_d(\mbf z)$ and $\mu_d(\mbf z)$ are the least and greatest element of $\ker f_n$, for 
$f_n : \mbf F(Z) \to \mbf{FD}(Z)$ the standard homomorphism with $f_n(z_i) = z_i$ for $i \in [1,n]$, where $Z = \{ z_1, \dots, z_n \}$.  

\begin{lm} \label{lm:3cases}  
Let $a \in \mbf A_n$ and let $h_n^{-1}(a)$ be the interval $[\beta_a(\mbf z), \alpha_a(\mbf z)]$ in~$\mbf F_n$.
\begin{enumerate}
    \item[(1)] If\/ $a$ is not a join of two atoms of $\mbf A_n$, nor the upper cover of the join of two atoms, nor the meet of two coatoms, nor the lower cover of the meet of two coatoms,
    then $\beta_a(\mbf z) = \nu_{g_n(a)}(\mbf z)$ and $\alpha_a(\mbf z) = \mu_{g_n(a)}(\mbf z)$.
    \item[(2a)] If\/ $a$ is a join of two atoms of $\mbf A_n$,
     then $\beta_a(\mbf z) = \alpha_a(\mbf z) = \nu_{g_n(a)}(\mbf z)$.
    \item[(2b)] If\/ $a$ is the meet of two atoms of $\mbf A_n$,
     then $\beta_a(\mbf z) = \alpha_a(\mbf z) = \mu_{g_n(a)}(\mbf z)$.
     \item[(3a)]  If\/ $a$ is the upper cover in $\mbf A_n$ of the join of two atoms, then $\beta_a(\mbf z)$ is the
     upper cover of the join of the corresponding two atoms of\/ $\mbf F_n$,
     and $\alpha_a(\mbf z) = \mu_{g_n(a)}(\mbf z)$.
     \item[(3b)] If\/ $a$ is the lower cover in $\mbf A_n$ of the meet of two coatoms, then $\alpha_a(\mbf z)$ is the
     lower cover of the meet of the corresponding two coatoms of\/ $\mbf F_n$,
     and $\beta_a(\mbf z) = \nu_{g_n(a)}(\mbf z)$.    
\end{enumerate}
\end{lm}

\smallskip
\begin{proof}

\noindent \underline{Case 1}: 
Assume $a$ is not a join of two atoms of $\mbf A_n$, nor the upper cover of the join of two atoms, nor the meet of two coatoms, nor the lower cover of the meet of two coatoms.
%The preimage $h_n^{-1}(a)$ is an interval $[\beta_a(\mbf z),\alpha_a(\mbf z)]$ in $\mbf F_n$.
Since in this case $a$ is not one of the doubled elements,
$h_n^{-1}(a) = f_n^{-1}(g_n(a))$ as sets, i.e.,
$h_n(w)=a$ iff $f(w)=g_n(a)$.
(Left to right always holds, right to left only for non-doubled elements.)    %%maybe omit this sentence?  
Thus for these elements the bounds on $h_n$ are the bounds on the standard homomorphism $f_n: \mbf F_n \to \mbf{FD}_n$. 
Choose $w \in \mbf F_n$ such that $h_n(w)=a$.
The lower bound $\beta_a(\mbf z)$ is the join normal form of $f_n(w) \in \mbf{FD}_n$, and the upper bound $\alpha_a(\mbf z)$ is the meet normal form of $f_n(w)$. 
For example, let $a$ be the upper cover of $z_1$ in Figure~\ref{fig:doubled}.  
Then, in $\mbf F_3$,  $\beta_a(\mbf z) = z_1 + z_2z_3$ and $\alpha_a(\mbf z) = (z_1 + z_2)(z_1+z_3)$,
which are the join and meet normal forms (respectively) of the upper cover of $z_1$ in $\mbf{FD}_3$.
% As a special instance of Case 1, if $a$ is an atom of $\mbf A_n$, then $h_n^{-1}(a)$ is the corresponding atom of $\mbf F_n$.  For example, for $a=z_1z_2z_3$ in $\mbf A_3$, we get $\beta_a(\mbf z) = \alpha_a(\mbf z) = z_1z_2z_3$ in $\mbf F_3$.
%Dually for coatoms.
%moved this one to preliminaries
%(Indeed, the following elements of $\mbf{FD}_n$ have a unique preimage under $f_n$:  0, 1, atoms, coatoms, generators.)

\noindent \underline{Case 2}: 
If $a$ is the join of two distinct atoms of $\mbf A_n$, say $a=b+c$, then $h_n^{-1}(a)$ is the join of the corresponding two atoms of $\mbf F_n$, and again we get $\beta_a(\mbf z)=\alpha_a(\mbf z)$.
That is because the join of two atoms in $\mbf F_n$ has a unique (strong) upper cover, and its image under $h_n$ is a different element of $\mbf A_n$; see Case~3.
For example, if $a = z_1z_2z_3 + z_1z_2z_4$ in $\mbf A_4$, then $\beta_a(\mbf z)=\alpha_a(\mbf z) = z_1z_2z_3 + z_1z_2z_4$ in $\mbf F_4$.
Dually for the meet of two coatoms.

\noindent \underline{Case 3}: 
If $a$ is the upper cover in $\mbf A_n$ of the join of two distinct atoms, say $a = (b+c)^*$, choose $u \in \mbf F_n$ such that $h_n(u)=a$.  Note that $f_n(u)=g_n(a) = g_n(b+c)$ since $g_n$ collapses the interval.  
%\gianluca{$a = (b+c)$? What is $a = (b+c)^*$?} 
Then $\beta_a(\mbf z)$ is the upper cover of the join of the corresponding two atoms in $\mbf F_n$,
and $\alpha_a(\mbf z)$ is the meet normal form in $\mbf F_n$ of $g_n(b+c) \in \mbf{FD}_n$. 
The upper cover of a join of two atoms in $\mbf F_n$ is given in Example~3.45 of \cite{FJN}, that is, letting $\ul z_i = \prod_{\ell \ne i} z_\ell$, for $i \ne j$, we have:
\[  \mbf F_n \models 0 \prec \ul z_i \prec \ul z_i + \ul z_j \prec \prod_{k \ne i,j} (\ul z_i + \ul z_j + \ul z_k) \]
so that the last expression gives the upper cover $(\ul z_i + \ul z_j)^*$ of $\ul z_i + \ul z_j$.
See Figure~3.5 of \cite{FJN}.

As an example of Case~3, let $a = (\ul z_4 + \ul z_3)^*$ in $\mbf A_4$.  
With $h=h_4$, we calculate the lower and upper bounds of the $\ker h$ congruence class of $a$ using the prescription of Theorem~2.4 in \cite{FJN}:
\begin{align*}
\beta_0 (\ul z_j) &= \ul z_j  \text{ for } 1 \leq j \leq 4, \\
\beta_0((\ul z_4 + \ul z_3)^*) &= z_1z_2 .
\end{align*}
Now $\ul z_j \in \op D_0(\mbf A_4)$, so $\beta (\ul z_j) = \ul z_j$  for those elements, but $(\ul z_4 + \ul z_3)^*$ has two minimal nontrivial join covers, so we continue:
\begin{align*}
\beta_1((\ul z_4 + \ul z_3)^*) &= z_1z_2  \cdot (\ul z_4 + \ul z_3 + \ul z_2)(\ul z_4 + \ul z_3 + \ul z_1)  \\
                 &=  (\ul z_4 + \ul z_3 + \ul z_2)(\ul z_4 + \ul z_3 + \ul z_1)                                
\end{align*}
since the latter terms are below $z_1$ and $z_2$, respectively.
As $(\ul z_4 + \ul z_3)^*$ is in $\op D_1(\mbf A_4)$,  $\beta_1(a)$ is the least element of its congruence class.
For the largest element $\alpha(a)$ of the congruence class, note that $a$ is the meet of 3 meet prime elements in $\mbf A_4$, whence we can use the simpler computation
\begin{align*}
\alpha((\ul z_4 + \ul z_3)^*) &= \alpha(z_1) \cdot \alpha(z_2) \cdot \alpha(z_3+z_4) \\
                 &= z_1z_2(z_3+z_4)  .                               
\end{align*}
Thus when $a$ is the upper cover of $z_1z_2z_3 + z_1z_2z_4$ in $\mbf A_4$, we get $h_4^{-1}(a)$ to be the following $\mbf F_4$-interval: 
\begin{align*}
  [\beta_a&(\mbf z),\alpha_a(\mbf z)] = \\
 &[ (z_1z_2z_3 +z_1z_2z_4+z_1z_3z_4 )(z_1z_2z_3+z_1z_2z_4+z_2z_3z_4 )  , z_1z_2(z_3+z_4)  ]
 \end{align*} 
 which agrees with the claims of part (3a) of the lemma.
Dual calculations yield the bounds for lower cover of the meet of two coatoms.

\smallskip \noindent 
Because  $\mbf{FD}_n$ and $\mbf A_n$ have $\op D$-rank (and dual $\op D$-rank)  0 and 1 respectively (see \cite{FJN, JN77, NaPa14}),
the algorithms for finding $\beta$ and $\alpha$ terminate with $\beta_1$ and $\alpha_1$, as in the preceding example.
\end{proof}
%%\gianluca{The referee says: This wasn’t clear. Maybe find a better description of what is being left to the reader.}  THE EXAMPLE HAS BEEN ITS DETAILS NOW
\medskip

Note for future reference that the standard splittings extend to $\mbf A_n$:
if $Z = Z_1 \dot\cup Z_2$, then $A_n = \uparrow\! (\prod Z_1) \ \dot\cup\, \downarrow\! (\sum Z_2)$.
That is, for all $a \in \mbf A_n$, $a \geq \prod Z_1$ or $a \leq\sum Z_2$, and not both.

\subsection{The set $\msc E$}
Let $\msc E = \{ E^1, \dots, E^N \}$ be the collection of all ${n+1}$-tuples  of the form:
\begin{itemize}
\item $E^\ell = \{ I_1^\ell, \dots, I_{n+1}^\ell \}$;
\item each $I_j^\ell$ is a class of $\ker h_n$;
%where $h_n: \mbf F_n(Z) \to \mbf A_n$ naturally, i.e., $h(z_j)=z_j$.
\item there exists in $\mbf F_n$ an independent set $X^\ell = \{ x_1^\ell, \dots, x_{n+1}^\ell \}$ with $x_j \in I_j^\ell$ for $j \in [1,n+1]$.
\end{itemize}
Thus $\op{Loc}(\mbf x,E^\ell, \mbf z)$ denotes the predicate of the third property.
Note this means that $X^\ell$ generates a copy of $\mbf F_{n+1}$ inside of $\mbf F_n(Z)$, by Corollary~1.12 of~\cite{FJN}.

We can think of $\msc E$ as follows.  
There are infinitely many independent sequences $\mbf x = (x_1, \dots, x_{n+1})$ in $\mbf F_n$,
each of which generates a sublattice isomorphic to $\mbf F_{n+1}$.   
However, each such sequence is located (i.e., $x_j \in I_j$ for $j \in [1,n+1]$) in one of finitely many sequences of intervals
$E = (I_1, \dots, I_{n+1})$ where each $I_j$ is a class of $\ker h_n$.
We record these possibilities as $\msc E = ( E^1, \dots, E^N )$
with $E^\ell = (I_1^\ell, \dots, I_{n+1}^\ell)$ for $1 \leq \ell \leq N$.
The collection $\msc E$ will be used in the formulation of the sentence $\pi^n$.

It is important to note that none of the intervals in a member of $\msc E$ are $h_n^{-1}(a)$ for $a$ being $0$, an atom, a join of two atoms, $1$, a coatom, or a meet of two coatoms. 
These elements, which lie in singleton $\ker h_n$-classes, are never a generator for a copy of $\mbf F_{n+1}$ in $\mbf F_n$.
For if $b$ and $c$ are atoms of $\mbf F_n$, then $\downarrow\!(b+c)$ is finite, whereas $\downarrow\! x_j$ is infinite in $\mbf F_{n+1}$.  
Again see Figure~3.5 of \cite{FJN}.
In other words, the predicates in the third clause are $\beta_a(\mbf z) \leq x_j \leq \alpha_a(\mbf z)$ with $a$ neither below the join of two atoms, nor above the meet of two coatoms. 
%%\gianluca{$\beta_a(\mbf z) \leq x_j \leq \alpha_a(\mbf z)$?} FIXED

An interval $I_j^\ell$ could however be the preimage of the upper cover of the join of two atoms, 
or dually the preimage of the lower cover of the meet of two coatoms, as these are infinite intervals in $\mbf F_n$.

Now we define $\pi^n$:
\begin{align*}                 
& \exists z_1, \dots, z_n \  t(z_1+\dots+z_n) \ \&\ b(z_1 \cdots z_n) \ \& \\
&  \forall x_1, \dots,  x_{n+1} \ [\, \op{NI}( x_1, \dots, x_{n+1}) \  \op{or}\  \op{OR}_{E^\ell \in \msc E} \op{Loc}(\mbf x,E^\ell, \mbf z) \   ]
\end{align*}

\begin{lm} \label{lm:hold4}
$\pi^n$ holds in $\mbf F_n$.
\end{lm}
\begin{proof}
    Take $z_1, \dots, z_n$ to be the generators of $\mbf F_n$.
    Then $\pi^n$ holds by the definition of $\msc E$.
\end{proof}

Now let us show that $\pi^n$ fails in $\mbf F_{n+1}(X)$. 
That is, we want to show that for any $n$-tuple $(c_1, \dots, c_n)$  in $\mbf F_{n+1}$ with $\prod c_i = 0$ and $\sum c_i=1$,
and using the standard generating set $x_1, \dots, x_{n+1}$ of $\mbf F_{n+1}$ (so that $\op{NI}(\mbf x)$ fails), 
there is some $j \in [1,n+1]$ such that $\beta_a(\mbf c) \leq x_j \leq \alpha_a(\mbf c)$ holds for no $a \in \mbf A_n$.
Thus every option $\op{Loc}(\mbf x, E^\ell, \mbf z)$ in $\pi^n$ will fail.

Now $\beta_a(\mbf z)$ and $\alpha_a(\mbf z)$ are polynomials
in $(z_1, \dots, z_n)$; they have evaluations in $\mbf F_{n+1}(X)$ at $(c_1, \dots, c_n)$ under the substitution map $s: \mbf F_n(Z) \to \mbf F_{n+1}(X)$ with $s(z_i) = c_i$ for all $i$.  
The notation $\beta_a(\mbf c)$ means $s(\beta_a(\mbf z))$.

As a consequence of join and meet primeness of the generators:

\begin{lm} \label{lm:dagwood}

If\/ $c_1, \dots, c_n$ are such that $\prod c_i = 0$ and $\sum c_i=1$ in $\mbf F_{n+1}$, then for each generator $x_j$ there exist $c_i$ and $c_k$ with $c_i \leq x_j \leq c_k$.
%%\[  (\dagger) \qquad  c_i \leq x_j \leq c_k   \]
\end{lm}

Notice that at least one of the inequalities $c_i \leq x_j$ must be strict, since $|C|=n <n+1 = |X|$, so indeed we must have $c_i < x_j$ for at least one $i$, or else $\prod c_i = 0$ would not hold.  Dually, $x_j < c_k$ for at least one $k$.
 
Now let:
\begin{align*}
A &= \{ c_k : c_k > x_j \text{ for some } j \in [1,n+1] \} ,\\
B &= \{ c_i : c_i < x_j \text{ for some } j \in [1,n+1] \} ,\\
E &= \{ c_k : c_j = x_j \text{ for some } j \in [1,n+1] \} ,\\
P &= \{ c_k : c_j  \parallel x_j \text{ for all } j \in [1,n+1] \}\, .
\end{align*}
%%\gianluca{The referee says: are the definitions of $E$ and $P$ correct? Should, for example, $E = \{c_k : c_k = x_j \text{ for some } j \in [1, n + 1] \}$ and similarly for $P$?}  FIXED
These sets are pairwise disjoint, because $X$ is an antichain, and $|A|+|B|+|E|+|P| = n$ 
with $|A| \geq 1$ and $|B| \geq 1$.
This is due to $A \ne \varnothing$, $B \ne \varnothing$ (because of what was observed above), and $A \cup B \cup E \cup P = \{ c_1, \dots, c_n \}$.

Let $A_Z = \{ z_k : c_k \in A \}$ etc.~ be the corresponding sets of variables. 

As always $\mbf F_n(Z) =\,  \downarrow\! (\sum B_Z) \ \dot\cup\, \uparrow\! (\prod (A_Z \cup E_Z \cup P_Z))$.
Choose $x_j$ such that $x_j \neq c_i$ for all $i \in [1,n]$.
Suppose $x_j \in I_j^\ell$ holds as part of some $\op{Loc}(\mbf x,E^\ell, \mbf z)$.
Then $\beta_a(\mbf c) \leq x_j \leq \alpha_a(\mbf c)$ holds in $\mbf F_{n+1}(X)$ for some $a \in \mbf A_n$, with $a$ strictly above the join of any two atoms and strictly below the meet of any two coatoms. 
Moreover, $\prod (A_Z \cup E_Z \cup P_Z) \nleq \beta_a(\mbf z)$ in $\mbf F_n$ since $s(\beta_a(\mbf z)) = \beta_a(\mbf c) \leq x_j$, while $s(\prod (A_Z \cup E_Z \cup P_Z)) 
= \prod (A \cup E \cup P) \nleq x_j$ because $x_j$ is meet prime. 
Therefore $\beta_a(\mbf z) \leq \sum B_Z$ by the split. 
Hence $a = h_n(\beta_a(\mbf z)) \leq h_n(\sum B_Z)$.
That implies $\alpha_a(\mbf z) \leq \alpha_{\sum B_Z}(\mbf z) = \sum B_Z$, by Lemma~\ref{lm:3cases} as $\sum B_Z$ is always in Case~1.
(In Cases~2b and~3b, $g_n(a)$ is the meet of two coatoms,
say $\ol z_i \cdot \ol z_j$, and $\ol z_i \cdot \ol z_j > \sum_{k \neq i,j} z_k$ in $\mbf{FD}_n$, which would be the only option for $\sum B_Z$.)
Hence $x_j \leq \alpha_a(\mbf c) = s(\alpha_a(\mbf z)) \leq s(\sum B_Z) = \sum B$, contradicting the join primeness of $x_j$ as $x_j \nleq c_i$ for all $c_i \in B$.

\begin{thm} \label{thm:main}
$\pi^n$ is a positive sentence that holds in $\mbf F_n$ and fails in $\mbf F_{n+1}$.
\end{thm}

\subsection{The positive theory of (some) relatively free lattices}\label{positive_other_varieties}

A direct adaptation of the proof shows that for any lattice
variety $\msc V$ with $\mbf A_n \in \msc V$, and $n \geq 3$, the  sentence $\pi^n$ holds in the relatively free lattice $\mbf F \msc V_n$ and fails in $\mbf F \msc V_{n+1}$.

The following facts are relevant.
\begin{itemize}
   \item[(i)]  The generators of $\mbf F \msc V_n$ are join and meet prime; indeed, this property is inherited from free distributive lattices, as $\msc D$ is contained in $\msc V$.
   \item[(ii)]  The split of $\mbf F \msc V (Z)$ into $\uparrow\! (\prod Z_1) \ \dot\cup\, \downarrow\! (\sum Z_2)$
   whenever $Z = Z_1 \dot\cup Z_2$ is a consequence of (i).   
   \item[(iii)]  Whitman's condition $(\mathrm W)$ no longer holds, so independent sets need not generate a copy of a (relatively) free lattice.  Instead, we just work with independence \emph{per se}.
   \item[(iv)]  The condition $\msc E$ used in $\pi^n$ assumes that we know, at least in principle, where the independent subsets of $\mbf F_n$ of size $n+1$ are located.  We no longer have Whitman's or Tschantz's theorem to guarantee their existence in certain intervals; for all we know, there could be none.  Having fewer (or no) independent sets in $\mbf F \msc V_n$ just shortens the sentence $\pi^n$; the argument remains valid. 
   \item[(v)]  No independent set $X$ of $\mbf F\msc V(n)$ of size $k>n$ contains the join two atoms.  To see this, first recall that any nontrivial lattice variety satisfies $\msc D \leq \msc V \leq \msc L$, where $\msc D$ denotes distributive lattices and $\msc L$ denotes all lattices.  Hence there are natural homomorphisms $\mbf F \msc L(Z) \to \mbf F \msc V (Z) \to \mbf F \msc D(Z)$.  But the atoms and joins of pairs of atoms are the same in $\mbf F \msc L(Z)$ and $\mbf F \msc D(Z)$, \emph{viz.}, $0 \prec \ul z_i \prec \ul z_i + \ul z_j$.
   So the same holds in $\mbf F \msc V(Z)$.  Thus if $u$ is the join of two atoms in $\mbf F \msc V(Z)$, i.e. $u = \ul z_i + \ul z_j$ with $i \ne j$, then $|\!\downarrow\! u \,| = 4$. 
   %in $\mbf F \msc V(Z)$.
   Meanwhile, if $X$ is an independent set, then meets of distinct subsets of $X$ are distinct.  If $|X| = k > n \ge 3$ and $x \in X$, then $|\!\downarrow\! x\,| \geq 2^{k-1} \ge 8$.
   Consequently, an independent subset of $\mbf F\msc V_n$ with more than $n$ elements cannot contain an atom or the join of two atoms.
   \item[(vi)]  Since $\mbf A_n \in \msc V$, there is a surjective homomorphism $h_n : \mbf F \msc V_n \to \mbf A_n$.
   This ensures that a join of two atoms $\ul z_i + \ul z_j$ 
   and its upper cover $\prod_{k \ne i,j} (\ul z_i + \ul z_j + \ul z_k)$ are distinct in $\mbf F \msc V_n$, because they evaluate distinctly under $h_n$ in $\mbf A_n$.
   \item[(vii)]  Because $\mbf A_n$ is a lower bounded lattice, any homomorphism from a finitely generated lattice to $\mbf A_n$ is a bounded homomorphism; see \cite[Theorem~2.13]{FJN}.
\end{itemize}

\section{First-order rigidity}

In this section the following hypotheses stand: $\mathbf{K} \equiv \mathbf{F}_n$ and $\mathbf{K}$ is finitely generated. As observed in \cite{NaPa14} we have that $\mathbf{F}_n$ is a prime model of its theory, so without loss of generality we can assume that $\mathbf{F}(X) = \mathbf{F}_n \preccurlyeq \mathbf{K}$ and that $\mathbf{K}$ is generated by a finite set $Y \supseteq X$. Recall that an element $a$ in a lattice $\mbf L$ is said to be join irreducible if $a = b \vee c$ implies that either $a = b$ or $a = c$. Notice also that being join irreducible is a first-order property. Recall that we denote by $\text{J}(\mathbf{K})$ the set of join irreducible elements in $\mathbf{K}$. Also, we sometimes abbreviate ``join irreducible'' with j.i. Finally, recall all the other notions and notations from Section~\ref{bounded_sec}.

%Define $u \overline{\mathbf{D}} v$ if $u = v$ or $u \mathbf{D}^k v$ for some $k > 0$. (That's the closure of $\mathbf{D}$ or $=$.)

\begin{lm}
Every element of\/ $\mbf K$ is the join of\/ $3$ or fewer elements from $\op{J}(\mathbf{K})$.
\end{lm}

\begin{proof}
The property of the lemma is first-order.
Let us show it holds in $\mbf F_n$, whence it will also be true in $\mbf K$.

\smallskip \noindent The crucial observation is that every interval of length at least 2 in $\mathbf{F}_n$ contains a join irreducible element.
If an interval in $\mbf F_n$ is infinite, then it contains a copy of $\mathbf F_\omega$ by Tschantz's Theorem (\cite[Theorem~9.10]{FJN}), and hence it contains infinitely many join irreducible elements.  The finite intervals of $\mbf F_n$ are classified in Theorems~7.5 and~7.10 of~\cite{FJN}, and from this description we find that every finite interval has length at most 3, and intervals of length 2 or 3 contain at least one join irreducible element.

\smallskip \noindent Consider now an element $w \in \mathbf{F}_n$ with $w = w_1 \join \dots \join w_m$ canonically. 
If $m \leq 3$, then the property of the lemma holds, so assume $m \geq 4$.
Set $a = w_1 \join w_2$ and $b = \Join_{3 \leq i \leq m} w_i$.
Each of the intervals $[a,w]$ and $[b,w]$ has length $\geq 2$, and thus contains a join irreducible element.
If the join irreducibles are $a' \in [a,w]$ and $b' \in [b,w]$, then $a' \join b' = w$ as desired.
\end{proof}

\begin{lm}\label{additional_assumption}
In addition to the assumptions at the beginning of this section we can assume that:
\begin{enumerate}[(1)]
\item every element of $Y$ is join irreducible;
\item if $y, z \in Y$ are incomparable, then $y \wedge z \notin Y$.
\end{enumerate}
\end{lm}

\begin{proof}
For each non-j.i. $y \in Y$, write it as a join of j.i.'s, say $y = \bigvee y_i$. In $Y$, replace $y$ by the set of $y_i$'s, and we still have a finite generating set. Then if $y, z, y \wedge z$ are all in $Y$ remove $y \wedge z$ and it is still a generating set.
\end{proof}

Recall the definition of $\text{D}_k(\mathbf{F}_n)$ from Section~\ref{bounded_sec}. We have that the join irreducibles in $\text{D}_k(\mathbf{F}_n)$ are in $X^{\wedge(\vee\wedge)^k}$ (cf.~\cite[pg.~28]{FJN}). Let $\sigma(n, k) = |X^{\wedge(\vee\wedge)^k}|$. 
Let $\delta_0 (w)$ be the property that $w$ is join prime. 
For $k > 0$, let $\delta_k(w)$ be the sentence saying that every join cover $w \leq r \vee s$ refines ($\ll$) to a join cover $x \leq \bigvee U$ with $|U| \leq \sigma(n, k)$ and $\delta_{k-1}(u)$ for all $u \in U$. Observe that $\delta_k(w)$ is a first-order property that holds of every $w \in \text{D}_k(\mathbf{F}_n)$. Notice also that $\delta_k(a) \leftrightarrow a \in X^{\wedge(\vee\wedge)^k}$ is in $\mathrm{Th}(\mathbf{F}_n)$, where the clause $a \in X^{\wedge(\vee\wedge)^k}$ is to be written out as a disjunction over all the possibilities (recall that $X^{\wedge(\vee\wedge)^k}$ is finite).
%%\gianluca{The referee says: $\delta_0(w)$ need to be defined.} FIXED

%(9) $\delta_0(a) \leftrightarrow a \in X^\wedge$ is in $\text{Th}(\mathbf{F}_n)$.

%(10) For $k > 0$, $\delta_k(a) \leftrightarrow a \in X^{\wedge(\vee\wedge)^k}$ is in $\mathrm{Th}(\mathbf{F}_n)$.

%The clause $a \in X^{\wedge(\vee\wedge)^k}$ is to be written out as a disjunction over all the possibilities.

    Recall the definition of the join dependency relation $u D v$ from \ref{def_join_dep_rel}.

\begin{lm}\label{lemma_5.3}
For every $k < \omega$, $\mathrm{Th}(\mathbf{F}_n)$ contains the following sentence:
$$\forall u \text{ j.i. }[\forall v \text{ j.i.}(u D v \; \rightarrow \; \delta_k(v)) \rightarrow \delta_{k+1}(u)].$$
\end{lm}

\begin{proof}
Recall that for every $k < \omega$ we have that $\mbf{F}_n \models \delta_k(a) \leftrightarrow a \in X^{\wedge(\vee\wedge)^k}$. The lemma thus states that, for $u \in \mathbf{F}_n$, if $J(u) \setminus \{u\} \subseteq X^{\wedge(\vee\wedge)^k}$, then $u \in X^{\wedge(\vee\wedge)^{k+1}}$, where $J(u)$ is as \cite[pg. 48]{FJN}.
\end{proof}

\begin{lm}\label{sent_12} The following is true in $\mathbf{F}_n$:
\begin{itemize}
\item if $u$ is join irreducible but not join prime, then there exists a join irreducible $v$ such that $u D v$.
\end{itemize}
\end{lm}

    \begin{proof} This follow from \cite[Theorem~2.50]{FJN}.
\end{proof}

Combining the previous lemmas we get:

\begin{cor}\label{crucial_cor}
If $u$ is a j.i. in $\mathbf{K} \setminus \mathbf{F}_n$, then either (a) or (b), where:
\begin{enumerate}[(a)]
\item there exists $v$ j.i. in $\mathbf{K} \setminus \mathbf{F}_n$ such that $u D v$;
\item $u D v$ for infinitely many $v$'s in $\mathbf{F}_n$ (and hence $v$'s of arbitrarily large $D$-rank).
\end{enumerate}
\end{cor}

\begin{proof}
If $u$ is j.i. but not in $\mathbf{F}_n$, then it is not join prime, since we have that $\delta_0(a) \leftrightarrow a \in X^\wedge$ is in $\mathrm{Th}(\mathbf{F}_n)$. Hence by \ref{sent_12} we get $u D v$ for some j.i. $v$. If there is a $k \geq 0$ such that every such $v$ satisfies $\delta_k(v)$, then $\delta_{k+1}(u)$ holds by \ref{lemma_5.3}. Then the fact that $\delta_k(a) \leftrightarrow a \in X^{\wedge(\vee\wedge)^k}$ is true in $\mbf{F}_n$ implies that $u \in X^{\wedge(\vee\wedge)^{k+1}} \subseteq \mathbf{F}_n$, a contradiction. Hence, either $u D v$ for a $v$ satisfying no $\delta_k$, in which case $v \notin \mathbf{F}_n$, or they all satisfy some $\delta_k$ but there is no bound on $k$.
\end{proof}

    \begin{theorem}\label{the_rigidity_hammer} Suppose that $\mbf{K}$ is lower bounded, then neither (a) nor (b) from \ref{crucial_cor} can happen and so necessarily $\mathbf{K} = \mathbf{F}_n$.
\end{theorem}

    \begin{proof} Suppose that $\mathbf{K}$ is lower bounded so $\text{D}(\mathbf{K}) = \mathbf{K}$ by \cite[Theorem 2.13]{FJN}. For the sake of contradiction, suppose that $\mathbf{K} \neq \mathbf{F}_n$. Recall that by \ref{additional_assumption} we have that $Y$ (the generating set of $\mathbf{K}$) contains at least one j.i. element $u_0$ which is not in $\mathbf{F}_n$. Suppose that (a) happens for $u_0$, then there exists j.i. $u_1 \in \mbf K \setminus \mathbf{F}_n$ such that $u_0 D u_1$. Now, $u_1$ satisfies the same assumptions of $u_0$, hence either (a) or (b) happens for $u_1$. And we can repeat our argument. All in all, either we find an infinite $D$ sequence made of elements $(u_i : i < \omega)$ starting at $u_0$, which contradicts $\text{D}(\mathbf{K}) = \mathbf{K}$, or, for some $i < \omega$, $u_i$ is as in case (b). So suppose that (b) happens for $u_i$, for some $i < \omega$. By \cite[Lemma 2.19]{FJN}, $\mathbf{K}$ has the minimal join cover refinement property (cf.~\cite[pg.~30]{FJN}). Now, if $a, b \in \text{J}(\mathbf{K})$ and $a \text{ D } b$, then $b$ is in a minimal nontrivial join cover of $a$, so $a \in \text{D}_k(\mathbf{K})$ implies $b \in \text{D}_{k-1}(\mathbf{K})$. Hence a D-sequence starting with $a$ can contain at most $\rho(a) + 1$ elements, where $\rho(a)$ denotes the D-rank of $a$. As we are assuming that $\text{D}(\mathbf{K}) = \mathbf{K}$, condition (b) from \ref{crucial_cor} cannot happen either.
\end{proof}

    \begin{proof}[Proof of \ref{first-order_rig_th}] This follows from the results in this section, in particular \ref{the_rigidity_hammer}.
    \end{proof}

\bibliographystyle{amsplain}
\bibliography{FOT_bib}
\end{document}